\documentclass[english, 12pt]{amsart}
\usepackage{amssymb}
\usepackage{amsfonts}
\usepackage{amscd}
\usepackage[T1]{fontenc}
\usepackage{color}

\definecolor{immi}{rgb}{0,.6,.1}

\long\def\change#1{{\color{blue} #1}}

\newbox\removebox
\newcommand\remove[2]{%
\setbox\removebox=\ifmmode\hbox{$#2$}\else\hbox{#2}\fi%
\leavevmode
\rlap{\textcolor{#1}{\vrule height0.8ex depth-0.6ex width\wd\removebox}}%
\box\removebox
}
\long\def\bigremove#1{%
\par\setbox\removebox=\vbox{#1}%
\vbox{%
\vbox to0pt{\hbox{\tikz\draw[color=blue,thick] (0,0) -- (\wd\removebox,-\ht\removebox)  (\wd\removebox,0) -- (0,-\ht\removebox);}}
\box\removebox
}
}

\usepackage{mathrsfs} 

\newcommand{\cCexp}{\cC^{\mathrm{exp}}}

\newcommand{\Vol}{\operatorname{Vol}}

\newcommand{\Loc}{{\mathrm{Loc}}}

\newcommand{\sgn}{\operatorname{sgn}}

\makeatletter
\def\RFss@@#1{\RF^*_{\!*#1}}
\def\RFss@_#1{\RFss@@{,#1}}
\def\RFss{\@ifnextchar_{\RFss@}{\RFss@@{}}}
\makeatother

\def\VF{\mathrm{VF}}

\def\VG{\mathrm{VG}}
\newcommand{\RF}{{\rm RF}}

\def\ac{{\overline{\rm ac}}}

\def\res{\operatorname{res}}

\def\11{{\mathbf 1}}

\def\CC{{\mathbb C}}

\def\NN{{\mathbb N}}

\def\QQ{{\mathbb Q}}
\def\RR{{\mathbb R}}

\def\ZZ{{\mathbb Z}}

\def\cC{{\mathscr C}}
\def\cD{{\mathcal D}}

\def\cF{{\mathcal F}}

\def\cL{{\mathcal L}}
\def\cM{{\mathcal M}}

\def\cO{{\mathcal O}}

\def\llp{\mathopen{(\!(}}

\def\rrp{\mathopen{)\!)}}

\newcommand{\C}{{\mathbb C}}
\newcommand{\K}{{F}}
\newcommand{\Z}{{\mathbb Z}}
\newcommand{\Q}{{\mathbb Q}}

\newcommand{\F}{{\mathbb F}}
\newcommand{\R}{{\mathbb R}}

\newcommand{\gal}{\operatorname{Gal}}

\newcommand{\ldp}{\cL_{\operatorname{DP}}}
\newcommand{\ldpo}[1]{\cL_{\operatorname{DP},#1}}

\newcommand{\bG}{{\bf G}}

\newcommand{\bT}{{\bf T}}

\newcommand{\bM}{{\bf M}}

\newcommand{\fg}{{\mathfrak g}}

\newcommand{\fh}{{\mathfrak h}}
\newcommand{\Lie}{\operatorname{Lie}}
\newcommand{\rss}{\mathrm{rss}}

\newcommand{\reg}{\mathrm{reg}}
\newcommand{\tf}{{C_c^\infty}}
\newcommand{\Ad}{\operatorname{Ad}}

\newcommand{\dG}{\underline{\bG}}
\newcommand{\dM}{\underline{\bM}}
\newcommand{\dg}{\underline{\fg}}
\newcommand{\domega}{\underline{\omega}}

\def\llp{\mathopen{(\!(}}

\def\rrp{\mathopen{)\!)}}

\newtheorem{thm}{Theorem}[subsection]
\newtheorem{lem}[thm]{Lemma}
\newtheorem{cor}[thm]{Corollary}
\newtheorem{prop}[thm]{Proposition}

\theoremstyle{definition}
\newtheorem{defn}[thm]{Definition}

\newtheorem{notn}[thm]{Notation}
\newtheorem{example}[thm]{Example}

\newtheorem{def-prop}[thm]{Proposition-Definition}
\newtheorem{def-theorem}[thm]{Theorem-Definition}
\newtheorem{def-lem}[thm]{Lemma-Definition}

\theoremstyle{remark}
\newtheorem{remark}[thm]{Remark}
\newtheorem{rem}[thm]{Remark}

\theoremstyle{plain}

\numberwithin{equation}{subsection}

  {\par\medskip\noindent #1\par\begingroup%
    \advance\leftskip by 1em\advance\rightskip by 1em}%
  {\par\endgroup}

\newcommand{\ord}{\operatorname{ord}}

\newcommand{\coliff}{\;:\Longleftrightarrow\;}

\begin{document}

\setcounter{tocdepth}{1} 

\author[Cluckers]{Raf Cluckers}
\address{Universit\'e de Lille, Laboratoire Painlev\'e, CNRS - UMR 8524, Cit\'e Scientifique, 59655
Villeneuve d'Ascq Cedex, France, and,
KU Leuven, Department of Mathematics,
Celestijnenlaan 200B, B-3001 Leu\-ven, Bel\-gium}
\email{Raf.Cluckers@univ-lille.fr}
\urladdr{http://rcluckers.perso.math.cnrs.fr/}
 
\author[Gordon]{Julia Gordon}
\address{Department of Mathematics, University of British Columbia,
Vancouver BC V6T 1Z2 Canada}
\email{gor@math.ubc.ca}
\urladdr{http://www.math.ubc.ca/$\sim$gor}

\author[Halupczok]{Immanuel Halupczok}
\address{Mathematisches Institut, HHU D\"usseldorf,
Universit\"atsstr. 1, 40225 D\"usseldorf,
Germany}
\email{math@karimmi.de}
\urladdr{http://www.immi.karimmi.de/en/}

\thanks{The author R.C. was supported by the European Research Council under the European Community's Seventh Framework Programme (FP7/2007-2013) with ERC Grant Agreement nr. 615722
MOTMELSUM, by the Labex CEMPI  (ANR-11-LABX-0007-01), and  would like to thank both the Forschungsinstitut f\"ur Mathematik (FIM) at ETH Z\"urich and the IH\'ES for the hospitality during part of the writing of this paper. J.G. was supported by NSERC, and I.H. was supported by the SFB~878 of the Deutsche Forschungsgemeinschaft. The authors thank the referee for valuable comments.}
 


\title[Uniform analysis and orbital integrals]{Uniform analysis on local fields and applications to orbital integrals}

\begin{abstract}
We study upper bounds, approximations, and limits for functions of motivic exponential class, uniformly in non-Archimedean local fields whose characteristic is $0$ or sufficiently large.
Our results together form a flexible framework for doing analysis over local fields in a field-independent way.
As corollaries, we obtain many new transfer principles, for example, for local constancy, continuity, and existence of various kinds of limits.
Moreover, we show that the Fourier transform of an $L^2$-function of motivic exponential class is again of motivic exponential class.
As an application in representation theory, we prove uniform bounds for the Fourier transforms of orbital integrals on connected reductive $p$-adic groups.
\end{abstract}

\maketitle


\section{Introduction}
\label{sec:intro}

One branch of analysis on non-Archimedean local fields $F$ deals with functions from (subsets of) $F^n$ to $\CC$.
In this paper, we develop a framework which makes it possible to carry out this kind of analysis
uniformly in the local field $F$, in a similar sense as algebraic geometry works in a field-independent way.
This extends the work pursued in \cite{CGH}, \cite[Appendix B]{ShinTemp}, \cite{CGH4}, and
\cite[Section 4]{GordonHales}.
One of our motivations is the program initiated by Hales to reformulate in a field-independent way the
entire theory of complex admissible
representations of reductive groups over local fields, \cite{Hales,Hales1}. We note that
in retrospect, one can see an inkling of the possibility of such a reformulation
already in one of the earliest books on the subject, see
\cite{GelfandGraevPyat}, p.~121.
This paper develops the framework in which
this program can be carried out. Another motivation is to enable further applications in the line of e.g.~\cite{CHL,YGordon,CCGordonS,CGH2,ShinTemp,GordonHales}.
In particular, in Section \ref{sec:app}, we develop uniform in $p$ bounds for the Fourier transforms of orbital integrals, normalized
by the discriminant, discussed in more detail below.

Using our framework has several direct implications. The most striking one is probably that
for any property that can be expressed uniformly in the field using the formalism, one has
a transfer principle like the one of Ax--Koshen/Ershov \cite{AK1,Ersov}, i.e., whether the property holds can be transferred between
local fields of positive and mixed characteristic (provided that the residue field characteristic is big enough).
Another direct implication is that if some value
 (obtained using the formalism) can be bounded for each local field $F$ independently, then
one obtains some precise information on how the bound depends on $F$, for free. An example of this approach, applied to Fourier transforms of orbital integrals, appears in Section \ref{sec:app}.

By ``doing analysis uniformly in several fields'', we mean that the analytic operations are carried out on certain abstract objects
which can be specialized to every specific field, yielding the familiar concrete objects. We take a very naive
approach to this: we simply define an abstract object to be the collection of its specializations. We do however
require that this collection is given in some uniform way (to be made precise below).
As an example, one kind of objects we work with are definable sets in the sense
of model theory. According to the above approach, we consider a definable set as a collection $(X_F)_F$ of sets $X_F \subset F^n$, all given by the same formula,
where $F$ runs over the local fields we are interested in.

All our uniformity results are valid for all local fields whose characteristic is $0$ or sufficiently big.
For this reason, it makes sense to think of two objects $(X_F)_F$ and $(X'_F)_F$ as being the same if they differ only
for $F$ of small positive characteristic. (Formally defining the objects this way would have been possible but unhandy.)
Note that in particular, we obtain uniformity results even for local fields which are arbitrarily ramified extensions of $\QQ_p$, for every $p$. This is new, compared
to most previous papers, and it builds on \cite{CHallp}.

The central objects of study are ``$\cCexp$-functions'', as in \cite{CHallp}, also called ``functions of $\cCexp$-class'' or ``of motivic exponential class''. (Those generalize the ``motivic exponential functions'' of \cite{CGH,CGH3,CGH4}, based on the ``motivic constructible exponential functions'' of \cite{CLexp}.) A $\cCexp$-function $f$ is a uniformly given collection of
functions $f_F$ from, say, $F^n$ to $\CC$ for every local field $F$.
As an example, given a polynomial $g \in \ZZ[x]$, we obtain a $\cCexp$-function  by setting $f_F(x) := |g(x)|_F$,
where $|\cdot|_F$ is the norm on the local field $F$ (in fact, this function lies in a smaller class of $\cC$-functions -- a similar class without the exponentials, also defined below). Another example can be built by
composing a nontrivial additive character on $F$ 
with a polynomial. In general,
the class of $\cCexp$-functions is defined as the $\CC$-algebra generated by those and some other functions,
which involve the valuation on $F$, certain exponential sums over the residue field, and additive characters $F \to \CC^\times$.
(The word ``exponential'' refers to the presence of those additive characters, which behave like exponentiation.)

The class of $\cCexp$-functions has intentionally been defined in such a way that it is closed under integration in the following sense.
Given any $\cCexp$-function $f = (f_F)_F$, say in $m+n$ variables,
there exists a $\cCexp$-function $g = (g_F)_F$ in the first $m$ variables such that for every local field $F$ of characteristic $0$ or $\gg1$,
we have
\[
\int_{y\in F^n} f_F(x,y) |dy| = g_F(x) \qquad\text{for all } x \in F^m.
\]
(Here, the integral is a usual Lebesgue-integral with respect to a suitably normalized Haar measure on $F^n$; for the moment, we assume that $f_F(x,\cdot)$ is integrable for each $F$ and each $x$.)
In this sense, we consider $g$ as a ``field-independent integral of $f$''.
In this paper, we prove that a whole zoo of analytic operations on
$\cCexp$-functions can be done in a similar, field-independent way.

For an operation which turns functions into other functions, being able to carry it out field-independently in this paper means
that the class of $\cCexp$-functions is closed under the operation, in the same sense as it is closed under integration as above.

Another type of question that we want to be able to answer field-independently is
whether a function has
a given (analytic) property (like being integrable, being continuous, being locally constant, existence of various kinds of limits).
This really becomes meaningful only in a parametrized version:
Given a $\cCexp$-function $f$ in, say, $m + n$ variables, we want to understand how the set
\[X_F := \{x \in F^m \mid f(x,\cdot) \text{ has the desired property}\}\]
depends on $F$.
Typically, the collection $X = (X_F)_F$ is not a definable set; however for most of the properties we will consider, we will prove that
$X = (X_F)_F$ is what we call a ``$\cCexp$-locus'': $X_F = \{x \in F^m \mid g_F(x) = 0\}$  for some $\cCexp$-function $g$.
Those loci already appeared in \cite{CGH}; in the present paper, they play a central role.

Even though $\cCexp$-loci are not definable sets, they are almost as flexible; for example, the collection of $\cCexp$-loci is closed under positive boolean combinations,
under the $\forall$-quantifier and under various kinds of for-almost-all quantifiers. This flexibility is inherited by our formalism:
It suffices to prove that $\cCexp$-functions and $\cCexp$-loci are closed under very few
``fundamental'' operations, to then obtain
a multitude of other operations essentially for free (i.e., by combining the
fundamental ones in various ways). Accordingly, this paper contains a few long and deep proofs (of the fundamental operations),
from which then many other results follow easily.

As stated at the beginning of this introduction, our formalism yields Ax--Kochen/Ershov like transfer principles.
More precisely, one has a general transfer principle for arbitrary conditions that can be expressed in terms of $\cCexp$-loci.
(This follows directly from the transfer principle for equalities in \cite{CLexp}.)
This means that we obtain a transfer principle for every analytic condition which can be expressed uniformly
in the sense of this paper. For example, given a
$\cCexp$-function $f$, it only depends on the residue field of $F$ whether $f_F$ is continuous/integrable/bounded/constant/locally constant/has a limit at $0$/etc., provided that the residue field characteristic is big enough.

\medskip

The outline of the paper is as follows. Section~\ref{sec:basic-def} contains the basic definitions that
are necessary to understand all results; in particular, $\cCexp$-functions and
$\cCexp$-loci are defined.

In Section~\ref{sec:locbasicop}, we recall some results from previous papers needed in this
paper. Here, Proposition~\ref{locbasicop} plays a key role, listing the most important operations
under which the collection of $\cCexp$-loci is closed.

In the remainder of Section~\ref{sec:intro}, we recall some more results from \cite{CGH,CHallp} (in
particular about integrability), we introduce one new basic ``operation'' on $\cCexp$-loci (namely eventual behavior;
Section~\ref{sec:eventual}), and we give some first example applications of the locus formalism.

The remaining uniform analytic operations are grouped by topic in the next two sections,
as follows.

Section~\ref{sec:bounds} contains the results about bounds and suprema. Given a bounded $\cCexp$-function $f$
(i.e., such that $f_F$ is bounded for each $F$), we obtain results about 
the dependence of the bound on $F$,
and also, for families of functions, 
the dependence of the bound on the family parameter.
A statement about a bound on $f(x)$ is the same as a statement about $\sup_x |f(x)|$,
so it would be natural to first prove that the class of
$\cCexp$-functions is closed under taking suprema, a result which would also be useful for many other purposes.
Unfortunately, given
a real-valued $\cCexp$-function $f(x,y)$, the supremum $\sup_y f(x,y)$ is, in general, not a
$\cCexp$-function in $x$. However, what we obtain instead is that the supremum
is ``approximately'' a $\cCexp$-function. For all applications of the supremum
in this paper, this turns out to be strong enough.

The main topic of Section~\ref{sec:limits} is limit behavior: pointwise limits, uniform limits and limits
with respect to the $L^p$-norm for various $p$. As applications, we prove results about continuity and about Fourier transforms.

Finally, in Section~\ref{sec:app}, we give an application of the results of Section~\ref{sec:bounds} to Fourier transforms of orbital integrals. The Fourier transform of an orbital integral of a reductive Lie algebra  $\fg$ is itself a locally integrable function on $\fg$, and it is known that when normalized by the square root of the discriminant, this function is bounded on compact sets. We prove that for a definable family of definable compact sets $\omega_\lambda$, the bound is of the form $q_F^{a+b\|\lambda\|}$, where
$q_F$ is the cardinality of the residue field of $F$, and $a$ and $b$ are constants that depend only on $\fg$ and the formulas defining the sets $\omega_\lambda$.

This systematic study of bounds for $\cCexp$-functions grew out of the similar study for motivic functions without the exponential that was initiated in order to answer a question that arose in an application of the Trace Formula,
\cite[Appendix B]{ShinTemp}.
Though we do not have such an application in mind for the Fourier transforms of orbital integrals,
we use this opportunity to establish this bound, both for possible future applications, and to illustrate applicability of this machinery.
In particular, now that the orbital integrals are proved to be motivic exponential functions (see Lemma~\ref{lem:can.orb} and the preceding paragraph), all the results
of Section \ref{sec:limits} on the speed of convergence are applicable in particular, to families of orbital integrals, giving for free the estimates by a negative power of $q_F$ for all kinds of sequences of orbital integrals that are known to converge to zero.  We do not pursue any specific results of this kind in this paper, but note that they should follow from the results of
Section \ref{sec:limits} and Lemma \ref{lem:can.orb} automatically. We hope that these results will be useful in applications of the Trace Formula in the spirit of \cite{ShinTemp}.

\subsection{Summary of the results}
\label{sec:summary}

Spread over the paper, we give many examples of results that can be deduced
from the basic ones using the formalism.
For the convenience of the reader, we here give a summary of all results (including those that were known before).
Those marked with $^\star$ are ``fundamental'' ones, i.e., those which need real work;
those without $^\star$ are the ones deduced using the formalism.

\smallskip

\noindent
\textbf{Operations on $\cCexp$-functions:}
Applying any of the following operations to a $\cCexp$-function yields a $\cCexp$-functions again:
\nopagebreak

\noindent
\begin{tabular}{@{}ll}
Integral & Theorem~\ref{thm:mot.int.}$^\star$
\\
Bound & Theorem~\ref{thm:fam}
\\
Approximate supremum & Theorem~\ref{thm:fam:gen}$^\star$
\\
Pointwise limit & Theorems~\ref{limits}, \ref{limits:basic}$^\star$
\\
Continuous extension & Proposition~\ref{prop.cont.ext}
\\
Uniform and $L^p$-limit & Theorem~\ref{L^pcom}$^\star$
\\
Fourier transform & Theorem~\ref{stab:four:II}
\end{tabular}

\medskip

\noindent
\textbf{Properties of $\cCexp$-functions:} The following properties of $\cCexp$-functions
are given by $\cCexp$-loci (2nd column) and can be transferred in the Ax--Kochen/Ershov way (3rd column):
\nopagebreak

{\small
\noindent
\begin{tabular}{@{}lll}
Constant & Corollary~\ref{cor.const} &  Corollary~\ref{trans.cons}
\\
Locally constant & Corollary~\ref{cor.loc.const} &  Corollary~\ref{trans.loc.cons}
\\
$L^1$-Integrable & Theorem~\ref{thm:mot.int.}$^\star$ & \cite[Thm 4.4.1]{CGH}
\\
Bounded & Theorem~\ref{thm:fam}$^\star$ & \cite[Thm 4.4.2]{CGH}
\\
Limit is $0$ & Theorems~\ref{thm:limits0}, \ref{thm:limits0:gen} & Corollary~\ref{trans:lim}
\\
\raggedright
(Pointwise) limit exists &  Thms~\ref{limits}, \ref{limits:gen}, \ref{limits:basic}$^\star$ &  Corollary~\ref{trans:lim}
\\
Continuous & Corollary~\ref{cor.cont} & Corollary~\ref{trans:lim}
\\
Limits exist: uniform, $L^2$, $L^\infty$ & Theorem~\ref{L^pcom}$^\star$ & Corollary~\ref{trans:funlim}
\\
$L^1$-, $L^2$-, $L^\infty$-integrable & Corollary~\ref{cor.finLp} & Corollary~\ref{trans:Lplim}
\\
$L^\infty$-norm equal to $0$ & Corollary~\ref{cor.almost} & Corollary~\ref{trans:Lplim}
\end{tabular}

}
\medskip

\noindent
\textbf{Operations on $\cCexp$-loci and properties of $\cCexp$-loci:}
Applying any of the following operations to a $\cCexp$-locus yields a $\cCexp$-locus again:
\nopagebreak

\noindent
\begin{tabular}{@{}ll}
Boolean combinations, $\forall$-quantification & Proposition~\ref{locbasicop}$^\star$
\\
Eventual behavior & Proposition~\ref{for:all:large:mu}$^\star$
\\
Almost everywhere quantification & Corollary~\ref{cor.almost}
\end{tabular}

\medskip

\noindent
\textbf{Consequences of being a $\cCexp$-function / a $\cCexp$-locus:}\nopagebreak

\noindent
\begin{tabular}{@{}ll}
Transfer principle for $\cCexp$-loci & Theorem~\ref{thm.trans}$^\star$
\\
Bounds are uniform in the field & Theorems~\ref{thm:presburger-fam}, \ref{thm:fam}
\\
Speed of convergence of pointwise limits & Thms~\ref{thm:limits0}, \ref{thm:limits0:gen}, \ref{limits:basic}$^\star$
\\
$L^p$-convergence $\Rightarrow$ pointwise convergence & Lemma~\ref{Lp-vs-pointwise}$^\star$
\end{tabular}
\medskip

\noindent
\textbf{Applications to Fourier transforms of orbital integrals:}\\
\nopagebreak
{\small
\begin{tabular}{@{}ll}
Orbital integrals are of $\cC$-class
(up to a $\cC$-constant) & Lemma \ref{lem:can.orb}
\\
Fourier transforms of orbital integrals are of $\cCexp$-class & Lemma \ref{lem:can.orb}
\\
Uniform bound for Fourier transforms of orbital integrals  &  Theorem \ref{thm:orb.int.bound}
\end{tabular}

}

\subsection{Basic definitions: functions and loci of $\cCexp$-class}
\label{sec:basic-def}

The main object of study in this paper is the class of ``$\cCexp$-functions''; by incorporating all finite field extensions of $\QQ_p$ for all $p$ as in \cite{CHallp}, they generalize the ``motivic exponential functions'' of \cite[Section 2]{CGH4} and \cite[Section 2]{CGH3}, all based on \cite{CLexp}.
We note that this class of functions differs also in another way from the class of motivic constructible exponential functions originally
defined in \cite{CLexp}: instead of the abstract motivic functions of \cite{CLexp}, we consider collections of functions, uniformly in local fields. This avoids issues related to what is called \emph{null-functions} in \cite{Casselman-Cely-Hales}.
We recall this framework and fix our terminology. 

\begin{defn}[Local fields]\label{AO}
Let $\Loc$ be the collection of all non-Archimedean local fields $F$ equipped with
a uniformizer $\varpi_F \in F$ for the valuation ring of $F$.
(So more formally, elements of $\Loc$ are pairs $(F, \varpi_F)$.)
Here, by a non-Archimedean local field we mean a finite extension of $\Q_p$ or of $\F_p\llp t\rrp$ {for any prime $p$}.

Given an integer $M$, let $\Loc_{M}$ be the collection of $(F,\varpi_F) \in \Loc$ such that
$F$ has characteristic either $0$ or at least $M$.

We will use the notation ``for all $F \in \Loc_{\gg 1}$'' to mean ``there exists an $M$ such that for all $F \in \Loc_{M}$''.
\end{defn}

Note that in various previous papers, the characteristic of the residue field $k_F$ is required to be sufficiently big.
In contrast, results in this paper only require $F \in \Loc_{\gg 1}$, which allows the characteristic of $k_F$ to be arbitrary if the characteristic of $F$ is zero. That this is sufficient builds on the results of \cite{CHallp}. On the other hand, it still seems far out of reach to treat local fields $F$ of small positive characteristic, given that
the model theory of such fields is not understood.

\begin{notn}[Sorts]
Given a local field $F \in \Loc$, we sometimes write $\VF_F$ for the underlying set of $F$, we write $\cO_F$ for the valuation ring of $F$, $\cM_F$ for the maximal ideal, $\RF_F$ for the residue field and $q_F$ for
the number of elements of $\RF_F$. The value group (even though always being $\ZZ$) will sometimes be denoted by $\VG_F$.
Moreover, we introduce a notation for the \emph{residue rings:} For positive integers $n$, set $\RF_{n,F} := \cO_F/n\cM_F$
(where $n\cM_F = \{na \mid a \in \cM_F\}$).
We write $\ord_F\colon \VF_F \to \VG_F \cup \{\infty\}$ for the valuation map, $\res_F\colon \cO_F \to \RF_F$ for the residue map and more generally
$\res_{n,F}\colon \cO_F \to \RF_{n,F}$ for the canonical projections.

We use ``$W \subset F^*$'' (or ``$W \subset \VF_F^*$'') as a shorthand notation for $W$ being a subset of $F^n$ for some $n \ge 0$,
and likewise ``$W \subset \Z^*$'' (or ``$W \subset \VG_F^*$'') and ``$W \subset \RFss_{F}$'',
the latter meaning that $W$ is a subset of a product $\prod_{i=1}^\ell \RF_{n_i,F}$, for some $\ell$ and $n_1, \dots, n_\ell$.
\end{notn}

Note that $n\cM_F$ is always a power of the ideal $\cM_F$, but which power it is depends on $n$ and on $F$. Namely, for $r$ the unique non-negative integer such that the order of $\varpi_F ^r$ equals $\ord n$, one has $n\cM_F = \cM_F^{1+r}$. 
In particular, we have $n\cM_F = \cM_F$ and hence $\RF_{n,F} = \RF_F$ whenever the residue field characteristic of $F$ does not divide $n$. This somewhat technical definition  
of the sorts is necessary to obtain the desired uniformity in $F$. In particular, since any formula uses only finitely many sorts, 
this implies that it suffices to exclude finitely many $p$ to achieve that the formula does not use any residue ring at all (except $\RF_F$).
This fits to the papers that exclude small residue characteristic but do not need residue rings. 

We use the same generalized Denef--Pas language $\ldp$ as in \cite[Section~2.2]{CHallp} and consider each $F \in \Loc$ as a structure in that language.
It is defined as follows.

\begin{defn}\label{defn.lang}
The language $\ldp$
has sorts $\VF$ (the valued field), $\VG$ (the value group) and $\RF_n$ for each $n \ge 1$ (the residue rings), with
the ring language on $\VF$ and on each $\RF_n$, the ordered abelian group language on $\VG$, the valuation map
$\ord\colon \VF \to \VG \cup \{\infty\}$ and generalized angular component maps $\ac_n\colon \VF \to \RF_n$ sending $x$ to $\res_n(\varpi_F^{-\ord(x)}x)$.
See \cite[Section~2.2]{CHallp} for more details.
\end{defn}

For some applications, it may be useful to have additional constants in the language (e.g.
for the elements of a fixed subring of $F$, or for the uniformizer $\varpi_F$).
Using some standard techniques from model theory,
all of our results can be suitably reformulated in such an enriched language;
see Appendix~\ref{sec:const} for details.

An $\ldp$-formula uniformly yields sets for all $F \in \Loc$.
As stated in the introduction, it will be convenient to consider a definable set
as the collection of the sets it actually defines in the fields $F$:

\begin{defn}[Definable sets and functions]\label{defset}
A collection
$X = (X_F)_{F \in \Loc_{M}}$ of subsets $X_F\subset \VF_F^*\times \RF^*_{*,F}\times \VG_F^*$ for some $M$ is called a \emph{definable set} if there is an
$\ldp$-formula $\varphi$ such that $X_F = \varphi(F)$ for each
$F$ in $\Loc_{M}$.

For definable sets $X$ and $Y$, a collection $f = (f_F)_{F \in\Loc_{M}}$ of functions $f_F:X_F\to Y_F$ for some $M$ is called a \emph{definable function} and denoted by $f:X\to Y$ if the collection of graphs of the $f_F$ is a definable set. (For this to make sense, we assume the $M$ of $f$ to be at least as big as the ones of $X$ and $Y$.)
\end{defn}

In reality, we are only interested in definable sets and functions for ``big $M$'',
i.e., we think of $X$ and $X'$ as being equal if
\begin{equation}\label{eq.defset}
X_F = X'_F \text{ for all } F \in \Loc_{\gg1}.
\end{equation}
However, for practical reasons, it is often easier
to work with representatives instead of introducing this equivalence relation formally.
Nevertheless, we will often implicitly enlarge $M$, for example calling $(X_F)_{F}$ definable
if $(X_F)_{F\in\Loc_{M}}$ is definable for big enough $M$.

Our definition of definable sets is the one which will be most convenient for this paper,
but note that there are various other possibilities. For example, if one were to
define a definable set to be the collection indexed only by the local fields of characteristic $0$,
one could recover (\ref{eq.defset}) by the Ax--Kochen/Ershov transfer principle.

We apply the typical set-theoretical notation to definable sets $X, Y$, e.g.\ writing
$X \subset Y$ (if $X_F \subset Y_F$ for each $F \in \Loc_{\gg 1}$), $X \times Y$, and so on, which may increase $M$ if necessary.
More generally, we will often omit the indices $F$ when the intended meaning is clear,
writing e.g. ``$\{x \in X \mid \text{some condition written without indices $F$}\}$'' to mean the corresponding collection of subsets of the $X_F$;
also, by ``for all $x \in X$, $f(x) = g(x)$'' we mean ``for all $F \in \Loc_{\gg 1}$ and all $x \in X_F$, $f_F(x) = g_F(x)$''.

Using definable functions and sets as building blocks, we introduce
``functions of $\cC$-class''; those will then be generalized to ``functions of $\cCexp$-class''.

\begin{defn}[Functions of $\cC$-class]\label{motfun}
Let $X = (X_F)_{{F\in \Loc_{M}}}$ be a definable set.
A collection $H = (H_F)_F$ of functions $H_F:X_F\to\RR$ is called \emph{a function of $\cC$-class} (or simply a \emph{$\cC$-function}) on $X$ if
there exist integers
$N$, $N'$, and $N''$, nonzero integers $a_{i\ell}$, definable functions $\alpha_{i}:X\to \ZZ$ and $\beta_{ij}:X\to \ZZ$,
and definable sets  $Y_i\subset X\times \RFss$ such that for all $F\in \Loc_{ {M}}$ and all $x\in X$
\begin{equation}\label{eq.motfun}
H(x)=\sum_{i=1}^N  \# Y_{i,x} \cdot  q^{\alpha_{i}(x)} \cdot \big( \prod_{j=1}^{N'} \beta_{ij}(x) \big) \cdot \big( \prod_{\ell=1}^{N''} \frac{1}{1-q^{a_{i\ell}}} \big),
\end{equation}
where $Y_{i,x} = \{y\in  \prod_{t=1}^{\ell_i} \RF_{n_{i,t},F} \mid (x,y)\in Y_{i}\}$, for some $\ell_i$ and $n_{i,t}$.

We write $\cC(X)$ to denote the ring of $\cC$-functions on $X$.
\end{defn}

In this definition, we already omitted lots of indices $F$ from the notation. Equation (\ref{eq.motfun}) e.g.\ really means
\[
H_F(x)=\sum_{i=1}^N  \# Y_{i,F,x} \cdot  q_F^{\alpha_{iF}(x)} \cdot \big( \prod_{j=1}^{N'} \beta_{ijF}(x) \big) \cdot \big( \prod_{\ell=1}^{N''} \frac{1}{1-q_F^{a_{i\ell}}} \big).
\]
(Note that each $Y_{i,F,x}$ is a finite set, so $\# Y_{i,F,x}$ makes sense.)

\begin{defn}[Additive characters]\label{psiu}
For any local field $F$, let $\cD_F$ be the set of the additive characters $\psi$ on $F$ that are trivial on the maximal ideal $\cM_F$ of $\cO_F$ and
nontrivial on $\cO_F$.
\end{defn}

Expressions involving additive characters of $p$-adic fields often give rise to exponential sums; that's what the ``exp'' stands for
in the definition below.

\begin{defn}[Functions of $\cCexp$-class]\label{expfun}
Let $X  = (X_F)_{ {F\in \Loc_{M}}}$ be a definable set.
A collection $H = (H_{F,\psi})_{F,\psi}$ of functions $H_{F,\psi}:X_F\to\CC$ for $F\in \Loc_{M}$ and $\psi\in \cD_F$ is called \emph{a function of $\cCexp$-class}
(or a \emph{$\cCexp$-function}) on $X$ if
there exist integers $N>0$ and $n_i\geq 1$, functions $H_i=(H_{iF})_F$ in $\cC(X)$, definable sets $Y_i\subset X\times \RFss$ and definable functions $g_i:Y_i\to \VF$ and $e_i:{Y_i}\to \RF_{n_i}$ for $i=1,\ldots,N$, such that for all $F\in \Loc_{M}$, all $\psi\in \cD_F$ and all $x\in X_F$
\begin{equation}\label{fexp}
H_{F,\psi}(x)=\sum_{i=1}^N   H_{iF}(x)\left( \sum_{y \in Y_{i,F,x}}\psi\left( g_{iF}(x,y)   + \frac{e_{iF}(x,y)}{n_i}     \right)\right),
\end{equation}
where $\psi(a+\frac{v}{n})$ for $a\in F$ and $v\in \RF_{n,F}$, by abuse of notation, is defined as $\psi(a+\frac{u}{n})$ with $u$ any element in $\cO_F$ such that $\res_n(u)=v$, which is well defined since $\psi$ is constant on $\cM_F$.
We write $\cCexp(X)$ to denote the ring of $\cCexp$-functions on $X$.
\end{defn}

In the same way as we omit indices $F$, we will also omit indices $\psi$ from the notation,
for example writing ``$\forall x\in X\colon H(x) = H'(x)$'' to mean
``$\forall F\in \Loc_{\gg1},\psi\in \cD_F,x\in X_F\colon H_{F,\psi}(x) = H_{F,\psi}'(x)$''.
Also, we will freely consider collections $X  = (X_F)_{F\in \Loc_{M}}$ as collections
indexed by $F$ and $\psi$ not depending on $\psi$.

Loci of vanishing of $\cCexp$-functions will play a key role in this paper.
To some extent, we will be able to apply the same reasoning as we do to the definable sets.
We develop some terminology for that.

\begin{defn}[Loci]\label{locset}
A \emph{locus of $\cC$-class} (or a $\cC$-locus) is a zero set $X$ of a function $f \in \cC(\VF^*\times \RFss\times \VG^*)$,
i.e., $X  = (X_F)_{F \in \Loc_{M}}$, where $X_F = \{x \in \VF_F^* \times \RFss_F \times \VG_F^* \mid f_F(x) = 0\}$.
Similarly, a \emph{locus of $\cCexp$-class} is a zero set of a function $f \in \cCexp(\VF^*\times \RFss\times \VG^*)$.
(The latter one is a collection indexed by $F$ and $\psi$).
\end{defn}

Note that any definable set is also a $\cC$-locus, since characteristic functions of definable sets are of $\cC$-class.
The converse, however, is not true; for example, $X = \{(x,y) \in \VG^2 \mid q^x - y = 0\}$ is a $\cC$-locus, but it is not definable.
Moreover, definable sets never depend on $\psi$.

Sometimes, it is notationally more convenient to refer to the condition describing a set instead of the set itself:

\begin{notn}[Conditions]\label{loccond}
Fix some integers $M,n,\ell,m_1, \dots, m_\ell,r$ and consider a collection $P = (P_{F,\psi})_{F \in \Loc_{M},\psi \in \cD_F}$ of
conditions $P_{F,\psi}(x)$ on elements $x \in \VF_F^n \times \RF_{m_1,F} \times \dots \times \RF_{m_\ell,F} \times \VG^r$.
We call $P$ a \emph{definable condition} if the collection of sets $X  = \{x \in \VF^* \times \RFss \times \VG^*
\mid P(x) \text{ holds}\}$ is definable.
Analogously, we call $P$ a condition of $\cC$-class or of $\cCexp$-class, respectively, if $X$
is a locus of the corresponding class. Again, we also say $\cC$-condition or $\cCexp$-condition for short.
\end{notn}

Thus a $\cCexp$-condition is a family $P = (P_{F,\psi})_{F,\psi}$ of conditions given by a $\cCexp$-function $f$, namely:
$P_{F,\psi}(x)$ holds if and only if $f_{F,\psi}(x) = 0$ (for all appropriate $F$, $\psi$, $x$).

Note that in the case $n = \ell = r = 0$, each $P_{F,\psi}$ is a condition on elements of a one point set
so it makes sense to ask whether $P_{F,\psi}$ holds without specifying any element.
(Definable such $P$ are essentially just first order sentences.)

\begin{rem}\label{remove-exp}
In the remainder of this section and in all of Sections~\ref{sec:bounds} and \ref{sec:limits}, literally every result that is stated for $\cCexp$-functions and $\cCexp$-loci is also valid if one replaces all occurrences of $\cCexp$ by $\cC$, i.e., if no input object of a result uses the additive character $\psi$, then $\psi$ does not appear in the output objects either. (This is obvious from the proofs.)
\end{rem}

\subsection{The locus formalism, transfer, and constancy}
\label{sec:locbasicop}

By definition of first order formulas, the class of definable conditions is closed under finite boolean combinations and under quantification.
Results from \cite{CGH,CHallp} state that this is partially also true for the new classes of conditions introduced in Notation~\ref{loccond}:
Each of them is closed under finite positive boolean combinations and under universal quantification.
Note however that unlike definable conditions, they are not closed under negation, and neither under
existential quantification. For $\cCexp$-conditions this follows from Example~\ref{ex.noExist}, and for $\cC$-conditions it follows from Example \ref{ex:existq}.

The following proposition summarizes the positive results for further reference:

\begin{prop}[Basic operations on Loci]\label{locbasicop}
In the following,
$x$ runs over a definable set $X$, and $y$ runs over a definable set $Y$.
\begin{enumerate}
 \item Any definable condition is a $\cCexp$-condition.
 \item If $P(x)$ and $Q(x)$ are $\cCexp$-conditions on $x$, then so are $P(x) \wedge Q(x)$ and  $P(x) \vee Q(x)$.
 \item If $P(x, y)$ is a $\cCexp$-condition on $x$ and $y$, then $\forall y\colon P(x,y)$ is a $\cCexp$-condition on $x$.
 \item A $\cCexp$-condition $P(x)$ stays a $\cCexp$-condition when considered as a condition on $x$ and $y$ that is independent of $y$.
\end{enumerate}
\end{prop}

Again, we are omitting indices. For example, ``$\forall y\colon P(x,y)$'' is the family $R = (R_{F,\psi})_{F,\psi}$ of conditions, where,
for $F\in \Loc_{\gg1},\psi\in \cD_F,x\in X_F$, $R_{F,\psi}(x)$ holds if and only if $P_{F,\psi}(x,y)$ holds
for all $y \in Y_F$.

\begin{proof}[Proof of Proposition~\ref{locbasicop}]
(1) is just a reformulation of the fact that definable sets are loci.
(2) is just a simple manipulation of the $\cCexp$-functions corresponding to $P$ and $Q$; see
Corollary 3.5.4 of \cite{CGH}.
(3) is the (Iva)-part of Theorem 4.4.2 of \cite{CHallp} (cf.\ also Theorem 4.3.2 of \cite{CGH}).
(4) follows from the corresponding statement for $\cCexp$-functions, i.e., that such functions in $x$
can also be considered as functions in $x$ and $y$.
\end{proof}

Another way of stating Proposition~\ref{locbasicop} is that any (syntactically correct) mathematical expression built out of $\ldp$-formulas,
``$f = 0$'' (for $f$ a $\cCexp$-function), $\wedge$, $\vee$ and $\forall$ defines a $\cCexp$-condition.
(Note that (4) is implicitly used when we write something like $P(x, y)\wedge Q(x, z)$, which is a condition on $x,y,z$.)
The $\cCexp$-functions $f$ in turn can be given by any expression of the form as in Definition~\ref{motfun} or \ref{expfun}.

Such expressions can be considered as a generalization of first order formulas, and they provide
a short and convenient way to prove that new conditions are of $\cCexp$-class.
Several of the main results of this paper (and also of \cite{CGH}) just state that certain
additional operations can be used in expressions defining $\cCexp$-conditions.
As an example of how this formalism works, we prove:

\begin{cor}[Constancy]\label{cor.const}
Let $f$ be in $\cCexp(X\times Y)$, for some definable sets $X$ and $Y$.
Then the set of $x \in X$ for which the function $y \mapsto f(x, y)$ is constant is a $\cCexp$-locus.
\end{cor}
\begin{proof}
That $f(x,\cdot)$ is constant can be written as $\cCexp$-condition as follows:
\[\forall y_1, y_2 \in Y\colon f(x,y_1) = f(x,y_2).\]
\end{proof}

To illustrate our conventions and formalism once more, let us give the following extra explanation of the above proof:
$f(x,y_1) - f(x,y_2)$ is a $\cCexp$-function in $x$, $y_1$, $y_2$, so this difference being zero defines a $\cCexp$-condition
on $x$, $y_1$, $y_2$, namely the family
$P = (P_{F,\psi})_{F,\psi}$, where $P_{F,\psi}(x, y_1, y_2)$ holds if and only if
$f_{F,\psi}(x,y_1) - f_{F,\psi}(x,y_2) = 0$. By Proposition~\ref{locbasicop} (3),
\[
 P'(x, y_1) \coliff \forall y_2\colon P(x, y_1, y_2)
\]
is a $\cCexp$-condition, and using Proposition~\ref{locbasicop} (3) once more,
so is
\[
P''(x) \coliff  \forall y_1\colon P'(x, y_1) \iff \forall y_1,y_2\colon P(x, y_1, y_2).
\]
This $P''(x)$ corresponds to a $\cCexp$-locus $Z = \{x \in X \mid P''(x)\}$, which, by definition of $P''$, is the one desired by the corollary:
For every $F, \psi$, we have $x \in Z_{F, \psi}$ if and only if $f_{F,\psi}(x,y_1) - f_{F,\psi}(x,y_2) = 0$ holds for all $y_1, y_2 \in Y_F$
(which just means that $f_{F,\psi}(x,\cdot)$ is constant).

One motivation to introduce $\cCexp$-conditions is that
for them, one has a transfer principle analogous to the Ax--Kochen/Ershov Theorem for first order sentences:
Whether a $\cCexp$-condition holds only depends on the residue field, provided that the residue field characteristic is big enough.
More precisely, we have the following:

\begin{thm}[Transfer]\label{thm.trans}
Suppose that $P(x)$ is a $\cCexp$-condition on $x \in X$, for some definable set $X$.
Then there exists an $M$ such that for all $F \in \Loc$ with
residue field characteristics at least $M$, the following holds.

If $P_{F,\psi}(x)$ holds for all $\psi \in \cD_{F}$ and all $x \in X_F$, then for any
$F' \in \Loc$ whose residue field is isomorphic to the one of $F$,
$P_{F',\psi}(x)$ holds for all $\psi \in \cD_{F'}$ and all $x \in X_{F'}$.
\end{thm}

\begin{proof}
This is Proposition~9.2.1 of \cite{CLexp}.
\end{proof}

Note that as a ``basic transfer result'', it would be enough to have Theorem~\ref{thm.trans} for 0-ary
$\cCexp$-conditions. Indeed, Theorem~\ref{thm.trans} for general $P(x)$ is obtained by applying transfer
to the 0-ary condition
\[
\forall x \in X\colon P(x),
\]
which is of $\cCexp$-class by Proposition~\ref{locbasicop} (3).

This transfer principle can be combined with many other results from this paper.
As a first example, we obtain:

\begin{cor}[Transfer for constancy]\label{trans.cons}
Let $X$ and $Y$ be definable sets, and let $f$ be in $\cCexp(X\times Y)$.
Then there exists $M$ such that, for any $F\in \Loc$ of residue field characteristic $\ge M$,
the truth of the following statement only depends on (the isomorphism class of) the residue field of $F$:
\[
\text{For all $\psi\in\cD_F$ and all $x\in X_F$, }
y\mapsto f_{F,\psi} (x,y) \text{ is constant on } Y_F.
\]
\end{cor}
\begin{proof}
By Corollary~\ref{cor.const}, $f(x,\cdot)$ being constant is a $\cCexp$-condition on $x$, so Theorem~\ref{thm.trans} applies.
\end{proof}

It might be tempting to replace our $\cCexp$-loci by a larger class which is additionally closed
under existential quantification, and maybe even under negation (thus turning it into a class of definable sets in a
first order language).
However, the following example shows that this would destroy the transfer principle.

\begin{example}\label{ex.noExist}
Transfer (in the style of Theorem~\ref{thm.trans}) does not hold for the statement
\begin{equation}\label{eq:no-trans}
\forall x \in \VF_F\colon \exists y \in \cO_F\colon \psi_F(x) = \psi_F(y).
\end{equation}
Indeed: The image $\psi(\cO_F)$ consists of the $p$-th roots of unity, where $p$ is the residue characteristic
of $F$. If $F$ itself has characteristic $p$, then the entire image $\psi(F)$ consists only of the $p$-th roots of
unity, whereas if $F$ has characteristic $0$, then $\psi(F)$ contains all $p^r$-th roots of unity for all $r > 1$.
Thus (\ref{eq:no-trans}) holds if and only if $F$ has positive characteristic.
\end{example}

\begin{example}\label{ex:existq}
We give an example that, in general,  $\cC$-loci are not closed under existential quantification.  
If $\cC$-loci would be closed under existential quantification, then in particular
$Y = \{y \in \VG \mid \exists x \in \VG\colon q^x - y = 0\}$ would be a $\cC$-locus, i.e., $Y = (Y_F)_F$ with
$Y_F = \{q_F^n \mid n \in \NN\}$. However, for any $H = (H_F)_F \in \cC(\VG)$ and any $F$, the zero set of $H_F$
is ultimately periodic, i.e., there exist $N, e \in \NN$ such that for $\mu > N$, whether $H_F(\mu)$ is zero or not only depends on $\mu \mod e$.
Indeed, this follows from Proposition~\ref{repar}: For $\mu$ big enough and in a fixed congruence class modulo some suitable $e$,
we have \[
H_F(\mu) = \sum_i c_{i}\mu^{a_i}q_F^{b_i\mu}
\]
for finitely many non-negative integers $a_i$, rational $b_i$ and complex $c_i$, and such a function is either constant equal to zero (namely when all $c_i$ are zero)
or has only finitely many zeros,
e.g. by \cite[Lemma~2.1.8]{CGH}. 
\end{example}

\subsection{Eventual behavior and local constancy}
\label{sec:eventual}

The first new building block for $\cCexp$-conditions provided by this paper is
the following proposition about eventual behavior.
(One should not confuse this result with Theorem 14 of \cite{GordonHales} where $\mu$ is not allowed to depend on $x$.)

\begin{prop}[Eventual behavior]\label{for:all:large:mu}
If $P(x, \mu)$ is a $\cCexp$-condition on $x$ running over a definable set $X$ and on $\mu$ running over $\VG$, then the condition $Q(x)$ given by
\begin{equation}\label{eq:event}
Q(x) \iff \exists \mu_0\in \VG\colon \forall \mu \ge \mu_0\colon P(x, \mu)
\end{equation}
is a $\cCexp$-condition on $x$. Moreover, $\mu_0$ can be chosen to depend definably on $x$, i.e.,
there exists a definable function $\mu_0:X\to \VG$ such that for each $F\in \Loc_{\gg1}$, each $\psi\in \cD_F$ and each
$x\in X_F$, if $Q_{F,\psi}(x)$ holds, then
$P_{F,\psi}(x,\mu)$ holds for all $\mu\geq \mu_{0,F}(x)$.
\end{prop}

Recall that (\ref{eq:event}) defines a condition depending not only on $x$, but implicitly also on $F$ and $\psi$,
that is, for any $F\in \Loc_M$, any $\psi\in \cD_F$, and any $x\in X_F$, the condition $Q_{F,\psi}(x)$ holds if and only if there is an integer $\mu_0$ (possibly depending on $F,\psi,x$) such that for all integers $\mu$ with $\mu \ge \mu_0$ one has $P_{F,\psi}(x, \mu)$.

Thus the first part of the proposition states that in expressions defining $\cCexp$-conditions (as explained below Proposition~\ref{locbasicop}),
we are also allowed to use the quantifier ``$\forall \lambda \gg 1$'' (meaning for all sufficiently big $\lambda \in \VG$).

The proof of this proposition (below) essentially follows a reasoning often used in \cite{CGH}, whose
main ingredient is that the dependence of $P$ on $\mu$ is controlled by Presburger-definable data.
However, given that we work in a context allowing arbitrary ramification, this becomes true only after
introducing additional $\RFss$-variables. The following proposition summarizes several of the
results from \cite{CHallp} we shall need.

\begin{prop}\label{repar}
Let $X$ and $U\subset X\times \VG$ be definable and let $H$ be a $\cCexp$-function on $U$. Then there exists
a definable bijection $\sigma\colon U \to U_{\rm par} \subset \RFss \times X \times \VG$ commuting with the projection to $X \times \VG$
and a finite definable partition of $U_{\rm par}$ such that we have the following for each part $A_j$.
We set $H_{\rm par} := H \circ \sigma^{-1} \in \cCexp(U_{\rm par})$.
\begin{enumerate}
 \item The set $A_j$ is a Presburger Cell over $\RFss \times X$, i.e.
\[
A_j =  \{(z,\mu) \in B_j \times \VG \mid \alpha_{j}(z) \leq \mu \le  \beta_{j}(z) \wedge \mu \equiv e_j\bmod n_j  \}
\]
where $e_j$ and $n_j > 0$ are integers, $B_j \subset \RFss \times X$
is the projection of $A_j$, $\alpha_j$ is either a definable function $B_j\to\VG$ or constant equal to $-\infty$ and
$\beta_j$ is either a definable function $B_j\to\VG$ or constant equal to $+\infty$.
\item If $\alpha_j \ne -\infty$, there exists an integer $N \ge 1$ and a definable function
$\alpha^0_j \colon X \to \VG$ such that $\alpha_j(\xi,x) - \alpha^0_j(x) \in [0, \ord(N)]$ for all $(\xi,x) \in B_j$;
and similarly for $\beta_j$. (And without loss, $N$ can be taken the same for all $\alpha_j$ and $\beta_j$.)
\item
 There are finitely many functions $c_{ij}$ in $\cCexp(B_j)$ and distinct pairs $(a_{ij},b_{ij})$ with $a_{ij}$ nonnegative integers and $b_{ij}$ rational numbers such that for all $F\in \Loc_{ \gg 1}$, for each $\psi \in \cD_F$ and each $(z,\mu)\in A_{j,F}$,
\begin{equation}\label{eq.repar}
H_{{\rm par}, F,\psi} (z,\mu) = \sum_{i} c_{ij,F,\psi}(z) \mu^{a_{ij}} q_F^{b_{ij}\mu}.
\end{equation}
\end{enumerate}
(The map $\sigma$ is called a reparameterization.)
\end{prop}

\begin{proof}
Let $f_i\colon U \to \VG$ be the definable functions into the value group appearing in the definition of $H$.
Apply \cite[Corollary~5.2.3]{CHallp} to $U$ and to those functions $f_i$.
This yields a reparameterization $\sigma\colon U \to U_{\rm par} \subset \RFss \times X \times \VG$ and a
finite definable partition
of $U_{\rm par}$ such that
each $f_{i,\rm par} := f_i \circ \sigma^{-1}$ is linear over $\RFss \times X$ on each part $A_j$,
i.e., $f_{i,\rm par}(z, \cdot)\colon (A_j)_z \to \VG$ is linear for each fixed $z \in \RFss \times X$, with the coefficient of $\mu$ independent of $F$ and $z$.
This already implies (3).

By refining the partition using Presburger cell decomposition, we obtain (1). Finally, we
apply \cite[Corollary~5.2.4]{CHallp} to the functions $\alpha_j$ and $\beta_j$ bounding the cells.
This yields a partition of each $B_j$ such that, after using that partition to refine the partition $A_j$,
(2) holds.
\end{proof}

\begin{proof}[Proof of Proposition~\ref{for:all:large:mu}]
It is enough to find a definable $\mu_0$ as in the ``moreover'' part. Indeed, using such a $\mu_0$,
$Q$ can be expressed as
\begin{equation}\label{eq:falm}
Q(x) \iff \forall \mu \ge \mu_0(x)\colon P(x, \mu),
\end{equation}
which is a $\cCexp$-condition by Proposition~\ref{locbasicop} (1), (2) and (3) (since it can be written as
$\forall \mu \in \VG\colon (\mu < \mu_0(x) \vee  P(x, \mu))$).

Let $H$ in $\cCexp(X \times \VG)$ be a function such that $P$ holds iff $H = 0$.
To obtain $\mu_0$, we apply Proposition~\ref{repar} (with $U = X \times \VG$); we also use the notation from there.

We may reduce to the case where $H_{\rm par}$ is non-zero only on a single part $A_j$. Indeed,
the function $H_j(x, \mu)$ which is equal to $H(x,\mu)$ iff $\sigma(x, \mu) \in A_j$ and $0$ otherwise
is of $\cCexp$-class, and we can define $\mu_0$ to be the maximum of the $\mu_{0,j}$ corresponding to the $H_j$.

If the upper bound $\beta_{j}$ defining $A_j$ is not $+\infty$, then using
Proposition~\ref{repar} (2) we can take $\mu_0(x) := \beta^0_j(x) + \ord(N) + 1$,
so now assume $\beta_{j} = +\infty$. In that case, we claim that we can take $\mu_0(x) = 0$.
More precisely, we claim that for each $x$, the set $C_x := \{\mu \mid H(x,\mu) \ne 0\}$ either is empty or has no upper bound.

From
\[
H_{{\rm par}} (z,\mu) = \sum_{i} c_{ij}(z) \mu^{a_{ij}} q^{b_{ij}\mu}
\qquad\text{for } (z,\mu) \in A_j
\]
(see (\ref{eq.repar})) and using that all pairs $(a_{ij},b_{ij})$ are distinct,
we obtain that for each $z \in B_j$, the set
\[C'_z := \{\mu \in \VG \mid (z,\mu) \in A_j \wedge H_{\rm par}(z,\mu) \ne 0\}\]
is empty if $c_{ij}(z) = 0$ for each $i$, and unbounded otherwise. Now the claim follows, using
that $C_x$ is a union of sets of the form $C'_z$.
\end{proof}

Using Proposition~\ref{for:all:large:mu}, we can turn global $\cCexp$-properties of functions
into local ones. As an example, we prove several variants of: local constancy is
a property of $\cCexp$-class.

\begin{cor}[Local constancy]\label{cor.loc.const}
Let $f$ be in $\cCexp(X\times Y)$, for some definable sets $X$ and $Y\subset \VF^n$ for some $n > 0$. Then we have the following
(where for each $F$, on $Y_F\subset F^n$ we put the topology induced by the one on $F^n$).
\begin{enumerate}
 \item The set of $x \in X$ for which the function $y \mapsto f(x, y)$ is locally constant is a $\cCexp$-locus.
 \item The set of $(x,y) \in X \times Y$ such that the function $f(x, \cdot)$ is constant on a neighborhood of $y$
     is a $\cCexp$-locus.
Moreover, the radius of constancy can be chosen definably,
i.e., there exists a definable function $\varphi:X\times Y\to \VG$ such that
for all $F \in \Loc_{\gg1}$, for all $\psi \in \cD_F$,
for all $x \in X_F$ and for all $y \in Y_F$, if $f_{F,\psi}(x, \cdot)$ is constant on a neighborhood of $y$, then it is constant on the intersection of $Y_F$ with the ball
of valuative radius $\varphi_F(x, y)$ around $y$ in $F^n$.
\end{enumerate}
\end{cor}

\begin{proof}
(1) follows from (2) and Proposition~\ref{locbasicop}, since $f(x, \cdot)$ being locally constant can be expressed by
\[
\forall y \in Y\colon \text{$f(x,\cdot)$ is constant on a neighborhood of $y$}.
\]

(2) That $f(x, \cdot)$ is constant on a neighborhood of $y$ can be expressed by
\[
\forall \lambda \gg 1\colon f(x,\cdot) \text{ is constant on $Y\cap B_\lambda(y)$},
\]
where $B_\lambda(y)$ is the ball around $y$ of valuative radius $\lambda$, which again
is a condition of $\cCexp$-class by Propositions~\ref{locbasicop} and \ref{for:all:large:mu} (and, if one wants,
Corollary~\ref{cor.const}).

Moreover, the witness $\mu_0$ for the quantifier ``$\forall \lambda \gg 1$'' provided by Proposition~\ref{for:all:large:mu}
is the desired definable function $\varphi$.
\end{proof}

In particular, we obtain a transfer principle for local constancy:

\begin{cor}[Transfer for local constancy]\label{trans.loc.cons}
Corollary~\ref{trans.cons} still holds if one replaces
``constant'' by ``locally constant''.
\end{cor}

\begin{proof}
Apply Theorem~\ref{thm.trans} to ``$f(x, \cdot)$ is locally constant'', which is a $\cCexp$-property (about $x$)
by Corollary~\ref{cor.loc.const}.
\end{proof}

\subsection{Integration}

We put the Haar measure on $F$ so that $\cO_F$ has measure $1$; on
$\RF_{m,F}$ and on $\ZZ$, we use the counting measure and for $X\subset \VF^*\times \RFss\times \VG^*$ we use the measure on $X_F$ induced by the product measure on $\VF_F^*\times \RFss_F\times \ZZ^*$. Likewise, we put the discrete topology on $\ZZ$ and on $\RFss_F$, the valuation topology on $F$, the product topology on $F^*\times \RFss_F\times \ZZ^*$, and the subset topology on $X_F$.

Maybe the most important aspect of $\cCexp$-functions is that they have nice and natural properties related to integration, see e.g.~the following theorem stating stability under integration, generalizing Theorem 9.1.4 of \cite{CLexp}, by including also $p$-adic fields of small residue field characteristic.

\begin{thm}[{Integration; \cite[Theorems 4.1.1, 4.4.2]{CHallp}}]\label{thm:mot.int.}
Let $f$ be in $\cCexp(X\times Y)$, for some definable sets $X$ and $Y$.
\begin{enumerate}
 \item The set of $x$ such that $y\mapsto f(x,y)$ is $L^1$-integrable over $Y$ is a $\cCexp$-locus.
 \item There exists a function $I$ in $\cCexp(X)$ such that
$$I(x) = \int_{y\in Y}f(x,y)\,|dy|$$
whenever $y\mapsto f(x,y)$ is $L^1$-integrable over $Y$. In more detail,
for every $F \in \Loc_{\gg1}$, every $\psi \in \cD_F$ and
every $x \in X_F$, if $y \mapsto f_{F,\psi}(x, y)$ is $L^1$-integrable, then $I_{F,\psi}(x) = \int_{y\in Y_{F,\psi}}f_{F,\psi}(x,y)\,|dy|$.

\end{enumerate}
\end{thm}

Note that Theorem~\ref{thm:mot.int.} immediately implies a generalized version, where the
set we are integrating over is allowed to depend on $x$: If $f$ is a
$\cCexp$-function on $W \subset X \times Y$ (and $W_x \subset Y$ denotes the fiber of $W$ at $x \in X$), then
$\int_{W_x} f(x, y)\,|dy| = \int_{Y} f(x, y)\cdot 1_W(x,y)\,|dy|$, where $1_W$ is the characteristic function of $W$.
Thus by applying the theorem to $f(x, y)\cdot 1_W(x,y)$, we obtain that $L^1$-integrability of $f(x,\cdot)$ on
$W_x$ is a $\cCexp$-condition, and the value of that integral (where it exists) is equal to $I(x)$ for some function $I$ of $\cCexp$-class.

In \cite{CGH}, we proved that the condition that $f(x,\cdot)$ is locally integrable is also a $\cCexp$-condition of the corresponding class.
This can be deduced from Theorem~\ref{thm:mot.int.} in a similar way as for local constancy (Corollary~\ref{cor.loc.const}).
However note also that since integrability on all small balls is equivalent to integrability on all balls, one doesn't even
need Proposition~\ref{for:all:large:mu}.

\section{Bounds and approximate suprema}
\label{sec:bounds}

\subsection{Results about bounds and approximate suprema}

One result of \cite{CGH} (generalized to the present context in \cite{CHallp}) is that a $\cCexp$-function being bounded is a $\cCexp$-condition
(see Theorem~\ref{thm:fam} (1)). In this section, we describe how bounds depend on the field $F$ and on additional parameters;
this generalizes results from Appendix B of \cite{ShinTemp}.
We give three main results on bounds, in increasing strength and complexity, with Theorem \ref{thm:fam:gen} containing the core result, with the hardest and longest proof of this paper.

\begin{thm}[Bounds]\label{thm:presburger-fam}
Let $H$ be in  $\cCexp(W \times \VG^n)$,
where $W$ is a definable set and $n\geq 0$.  Then
there exist an integer $b$ and a definable $d \in \VG$ such that for all $F\in \Loc_{ \gg1}$, all $\psi \in \cD_F$ and all $\lambda\in \Z^n$, the following holds.

If the function
$$
w\mapsto |H_{F,\psi} (w,\lambda) |_{\C}
$$
is bounded on $W_F$, then one actually has
\begin{equation}\label{upper}
| H_{F,\psi} (w,\lambda) |_{\C}  \le q_F^{ b \|\lambda\| +  d} \mbox{ for all } w \in  W_F,
\end{equation}
where $\|\lambda\| = \sum_{i=1}^n |\lambda_i|_\CC$ and where $|\cdot|_\C$ is the norm on $\C$.
\end{thm}

By a definable $d \in \VG$, we mean a definable set $d=(d_F)_{F\in\Loc_M}$ where $d_F$ is a singleton in $\VG_F$ for each $F$, where we identify $d_F$ with the integer $d_F^0$ with $d_F=\{d_F^0\}$. Since we do not use parameters in our language until the appendix, a definable $d\in\VG$ can be bounded by $a+\ord (c)$ for some integers $a$ and $c \ge 1$ depending on $d$ but not on $F$, that is, $d_F\leq a+\ord_F (c)$ for each $F$; indeed, this follows from quantifier elimination \cite[Theorem 5.1.1]{CHallp}. If one is moreover only interested in $F$ of sufficiently big residue characteristic, then
one may assume $\ord_F (c)$ to be zero, so that the exponent in (\ref{upper}) becomes $b\|\lambda\| + a$ for some integers $a,b$ not depending on $F$.

In the appendix, $d$ can be more general depending on the constants added to $\ldp$, and on the requirements (axioms) put on the interpretations of the constant symbols.

We observe that in the case with $n=0$, and still using that we do not use parameters,
Theorem \ref{thm:presburger-fam} yields that
for any $\cCexp$-function $H$ on $W$, there exist integers $a, c \ge 1$ such that
for any $F$ (in $\Loc_{\gg1}$) and any $\psi$ for which $H_{F,\psi}$ is bounded, the bound can be taken to be $q_F^{a + \ord_F(c)}$.
Note also that $\ord_F(c)$ is relevant only when one is interested in uniformity in $F$ for small residue characteristic and
arbitrary (but finite) ramification.

Part (2) of the following theorem generalizes Theorem~\ref{thm:presburger-fam}
by allowing the domain of $\lambda$ to be an arbitrary definable set $X$ (instead of just $\VG^n$).
In part (1), we recall the result from \cite{CGH,CHallp} about boundedness, since it fits nicely with part (2).

\begin{thm}[Bounds]\label{thm:fam}
Let $H$ be in $\cCexp(W \times X)$,
where $W$ and $X$ are definable sets.
\begin{enumerate}
\item
The set of $x \in X$ such that the function $w\mapsto | H(w,x) |_{\C}$ is bounded is a $\cCexp$-locus.
\item
There exists a definable function
$\alpha:X\to\VG$ such that for every $F \in \Loc_{ \gg 1}$, every $\psi \in \cD_F$ and every
$x \in X_F$, if the function
$$
w\mapsto | H_{F,\psi}(w,x) |_{\C}
$$
is bounded on $W_F$, then one has
$$
| H_{F,\psi}(w,x) |_{\C}  \le q_F^{\alpha_F(x)}  \mbox{ for all } w \in  W_F.
$$
\end{enumerate}
\end{thm}

This theorem implies Theorem~\ref{thm:presburger-fam} as follows.

\begin{proof}[Proof of Theorem \ref{thm:presburger-fam}]
Applying Theorem \ref{thm:fam} (2) with $X = \VG^n$ yields a definable function $\alpha:X\to\VG$. To obtain Theorem~\ref{thm:presburger-fam},
we need to know that we can replace $\alpha_F(\lambda)$ by $b\|\lambda\| + d_F$ for some definable $d$ and some integer $b$.

Note that when allowing big ramification, a definable function $\alpha$ is not necessarily piecewise linear.
However, by \cite[Corollary~5.2.5]{CHallp}, and up to working piecewise with finitely many definable pieces $X_i \subset X$, it differs from a linear definable function $\ell\colon X_i \to \VG$ by at most $\ord(c)$ for some integer $c \ge 1$. More precisely, for every $F \in \Loc_{\gg 1}$ and every $\lambda \in X_{i,F}$, $|\alpha_F(\lambda) - \ell_F(\lambda)| \le \ord_F(c)$,
where $c$ does not depend on $F$.
Moreover, we may assume that the coefficients of $\ell_F$ (which are rational numbers) are independent of $F$ and the constant term is given by a definable constant $d \in \VG$.

Thus on each piece $X_i$, we obtain a bound of the form $b\|\lambda\| + d$. Finally we take maxima of the so-obtained finite collection of $b$ and $d$, yielding a bound of the desired form.
\end{proof}

In general, the supremum over $w$ of $| H(w,x) |_{\C}$ is not a $\cCexp$-function in $x$.
However, it can be reasonably well approximated by the square root of a $\cCexp$-function.
The precise statement is the following.

\begin{thm}[Approximate suprema]\label{thm:fam:gen}
Let $H$ be in $\cCexp(W \times X)$,
where $W$ and $X$ are definable sets.
Then there exist a definable $d \in \VG$ and a function $G$ in $\cCexp(X)$
taking non-negative real values such that
the following holds: For each $F \in \Loc_{\gg1}$, each $\psi \in D_F$ and each $x \in X_F$ for which
the function
\begin{equation*}
w\mapsto  H_{F,\psi}(w,x)
\end{equation*}
is bounded on $W_F$, one has
\begin{equation}\label{eq:sq1}
\sup_{w\in W_F} |H_{F,\psi}(w,x)|_\C^2 \leq G_{F,\psi}(x) \le q_F^{d_F} \sup_{w\in W_F} |H_{F,\psi}(w,x)|_\C^2.
\end{equation}
\end{thm}

The remark about the definable $d\in\VG$ just below Theorem \ref{thm:presburger-fam} also applies to the $d$ of Theorem \ref{thm:fam:gen}.
We do not believe that Theorem \ref{thm:fam:gen} would hold without squaring, i.e.,
with $\sup_{w\in W} |H(w,x) |_{\C}$ instead of $\sup_{w\in W} |H(w,x) |^2_{\C}$.
Indeed, using a one-point set as $W$, this would imply that the absolute value of a
$\cCexp$-function can be approximated by a $\cCexp$-function, which already seems
unlikely for the function $H \in \cCexp(\VG^2)$, $H(x,y) = x^2 - y$.

Theorem \ref{thm:fam:gen} implies Theorem~\ref{thm:fam} as follows.

\begin{proof}[Proof of Theorem \ref{thm:fam}]
(1) is the Bdd-part of \cite[Theorem~4.4.2]{CHallp}.

(2) Apply Theorem \ref{thm:fam:gen} to $|H|^2_{\C}$ to obtain a (real-valued) function $G$ in $\cCexp(X)$ such that
$\sup_{w\in W} |H(w,x) |^2_{\C}\leq G(x)$ whenever $H(\cdot,x)$ is bounded.
It remains to show that any $G$ can be bounded by $q^{\alpha(x)}$
for some definable $\alpha$ (as functions depending on $F \in \Loc_{\gg1}, \psi \in \cD_F, x \in X_F$).
(We will more generally bound $|G(x)|_{\C}$ for arbitrary $G \in \cCexp(X)$, i.e., not necessarily real valued.)

Using the definition of $\cCexp$, one easily reduces to the case $G \in \cC(X)$. Indeed, by that definition,
\[
G(x)=\sum_{i=1}^N G_{i}(x)\Big( \sum_{y \in Y_{i,x}}\psi\big( g_{i}(x,y)  { + e_{i}(x,y)/n_i }     \big)\Big)
\]
for some $G_i\in \cC(X)$, some definable sets $Y_i \subset X \times \RFss$ and some definable functions $g_i:X\to \VF$ and $e_i:X'\to \RF_{n_i}$.
The complex norm of the inner sum is bounded by a product of some $\#\RF_{m_{ij}} = q^{1+\ord(m_{ij})}$ for each $i$, so it remains to bound the $|G_i(x)|_{\C}$.

So now suppose $G \in \cC(X)$. By definition of $\cC$, we have
\[
G(x)=\sum_{i=1}^N  \# Y_{i,x} \cdot  q^{\alpha_{i}(x)} \cdot \big( \prod_{j=1}^{N'} \beta_{ij}(x) \big) \cdot \big( \prod_{\ell=1}^{N''} \frac{1}{1-q^{a_{i\ell}}} \big)
\]
for some definable $Y_i \subset X \times \RFss$, $\alpha_i, \beta_{ij}\colon X \to \Z$ and some non-zero integers $a_{i\ell}$.
Clearly, each of the factors of the terms of the sum can be bounded as desired (using again $\#\RF_{m} = q^{1+\ord(m)}$), and hence so can the sum.
\end{proof}

\subsection{Proof of Theorem \ref{thm:fam:gen}}\label{theproofs}

The proof of Theorem \ref{thm:fam:gen}
is based on two deep
results, which we develop first:
one which allows us to reduce to functions living on $\RFss$ and $\VG$, and
another one which allows us to find approximate suprema of functions on $\VG$.

A key result of \cite{CGH} is its Proposition~4.5.8; the following is the generalization
of that result to the present context given in \cite{CHallp}. (Here, for simplicity, we
state a slightly weakened version.)

\begin{prop}[{\cite[Proposition~4.6.1]{CHallp}}]
\label{IML}
Let $m\geq 0$ be an integer, let $X$ and $U \subset X\times \VF^m$ be definable,
and let $H$ be in $\cCexp(U)$. Write $x$ for variables running over $X$ and $y$ for variables running over $\VF^m$.
Then there exist integers $N\geq 1$, $d\geq 0$, a definable surjection $\varphi:U\to V\subset X\times \RFss\times \VG^*$ over $X$, definable functions $h_{i}:U\to \VF$, and functions $G_{i}$ in $\cCexp(V)$ for $i=1,\ldots,N$,
such that the following conditions hold for each $F\in \Loc_{\gg1}$ and each $\psi\in \cD_F$;
we omit the indices $F,\psi$ to keep the notation lighter.
\begin{enumerate}
 \item One has
 $$
 H(x,y) = \sum_{i=1}^{N}G_{i}(\varphi(x,y))\psi(h_{i}(x,y));
 $$
 \item if one sets, for $(x,r)\in V$,
$$
U_{x,r} := \{y\in U_{x}\mid \varphi(x,y)=(x,r)\}
$$
 and
 $$
 W_{x,r}:=\{y \in U_{x,r}\mid
 \max_{i} | G_{i}(x,r)|_{\CC} \leq |H(x,y)|_{\CC} \},
 $$
 then
\[
\Vol(U_{x,r}) \leq q^d \cdot \Vol(W_{x,r}) < +\infty,
\]
where the volume $\Vol$ is taken with respect to the Haar measure on $\VF^m$.
\end{enumerate}
Here, as usual, $U_{x}$ is the set of $y\in \VF^m$ such that $(x,y)$ lies in $U$, and, for the map $\varphi$ to ``be over $X$'' means that $\varphi$ makes a commutative diagram with the natural maps to $X$.
\end{prop}
Note that (2) in particular implies that if $U_{x,r}$ has positive measure, then
\begin{equation}\label{eq.max<sup}
\max_{i} | G_{i}(x,r)|_{\CC} \leq \sup_{y\in U_{x,r}}|H(x,y)|_{\CC}.
\end{equation}


By virtue of Proposition \ref{IML}, one can simplify the domain of a $\cCexp$-function, while preserving most of its behavior on growth and size, as follows.

\begin{cor}\label{from.exp.to.e.IML}
Let  $W$, $X$, and $U\subset W\times X$ be definable sets and let $H$ be in $\cCexp(U)$.
Then there exists a nonnegative integer $N$, a definable surjection
 $$
\varphi:U\to V\subset X\times \RFss \times \VG^*
 $$
over $X$ and a non-negative real valued function $\tilde H$ in $\cCexp(V)$ such that
for all $F\in \Loc_{ \gg1}$, $\psi \in \cD_F$, $x\in X_{F}$ and $v \in V_{F,x}$, one has
\begin{equation}\label{eq.tildeH}
\frac{1}{N} \tilde H_{F,\psi}(v) \leq \sup_{\substack{ w \in W_F\\  \varphi_F(w,x)=v } } |H_{F,\psi}(w,x)|^2_{\C} \leq N  \tilde H_{F,\psi}(v),
\end{equation}
where $V_x$ is the fiber of $V$ above $x$, and where by the right hand inequality, we in particular mean that the supremum is finite.
\end{cor}

\begin{proof}[Proof of Corollary~\ref{from.exp.to.e.IML}]
Write $W\subset \VF^m\times W_0$ for some $m$ and some $W_0 \subset \RFss\times \VG^*$. If $m=0$ there is nothing to prove. We proceed by induction on $m$.
Apply Proposition~\ref{IML} to $H$ (where the $X$ from the proposition is $X \times  W_0$), yielding integers $N,d$, a definable surjection $\varphi: U\to V\subset X\times \RFss \times \VG^*$ and $G_i$ in $\cCexp(V)$ for $i=1,\ldots,N$.
We may assume that each fiber of $\varphi$ has positive measure, by treating
the fibers of measure zero separately, using induction on $m$;
in particular, (\ref{eq.max<sup}) holds, i.e., $\max_i |G_i(v)|_{\C} \le \sup_{w, \varphi(w,x)=v} |H(w, x)|_{\C}$.
Now set
$$
\tilde H(v) := \sum_{i=1}^{N} |G_{i}(v)|_\CC^2 \quad \mbox{for $v$ in } V.
$$
Then $\tilde H$ is as desired (with this $N$): The left hand inequality of (\ref{eq.tildeH}) of follows from (\ref{eq.max<sup}),
and the right hand inequality follows from $H(w, x) = \sum_i G_i(\varphi(w, x))\psi(h_i(w,x))$ and $|\psi(h_i(w,x))| = 1$.
\end{proof}

We now come to the second ingredient to the proof of Theorem \ref{thm:fam:gen}.
Suppose that $H$ is a bounded $\cCexp$-function with domain in the value group, say $H \in \cCexp(W)$ for some definable $W \subset \VG$.
To obtain an approximate maximum of $H$,
we will choose a finite subset $W_0 \subset W$ such that $H$ takes its maximum on $W_0$, up to some factor $m$. We will need to be able to
do this in families, in such a way that $m$ and $\#W_0$ do not depend on the parameters, and the elements of $W_0$ depend definably on the parameters.
After various simplifications, what we end up needing are the following two lemmas:
Lemma~\ref{lem:Z-ubd} which is used in the case where $W$ is infinite, and Lemma~\ref{lem:Z-bd} which is used in the case where $W$ is finite but growing in size (when varying the family members).

\begin{lem}\label{lem:Z-ubd}
Suppose that for $i = 1, \dots, k$, we have $a_i \in \NN$ and $b_i \in \ZZ$. Then there exist positive integers $m$ and $\ell$
such that the following holds for every tuple $c = (c_1, \dots, c_k) \in \CC^k$ and every $q \ge 2$.
Suppose that the function
\[
h(w) := \sum_i c_i w^{a_i} q^{b_i w}
\]
is bounded on $\NN$; then
we have \[\sup_{w \in \NN} |h(w)| \le m\cdot \max_{0 \le w \le \ell} |h(w)|.\]
\end{lem}

\begin{lem}\label{lem:Z-bd}
Suppose that for $i = 1, \dots, k$, we have $a_i \in \NN$ and $b_i \in \ZZ$. Then there exist positive integers $m$ and $\ell$, and
Presburger definable functions $w_1, \dots, w_\ell\colon \NN \to \NN$ with $w_i(t) \le t$ for all $i$ and $t$
and such that the following holds for every tuple $c = (c_1, \dots, c_k) \in \CC^k$, every $t \in \NN$ and every $q \ge 2$.
Consider the function
\[
h(w) := \sum_i c_i w^{a_i} q^{b_i w} \quad\mbox{for } w\in\NN;
\]
then we have \[\max_{0 \le w \le t} |h(w)| \le m\cdot \max_{1 \le i \le \ell} |h(w_i(t))|.\]
\end{lem}

We will prove both lemmas together.

\begin{proof}[Proof of Lemmas \ref{lem:Z-ubd} and \ref{lem:Z-bd}]
We start working on Lemma~\ref{lem:Z-bd}. First note that we may impose lower bounds on $t$, by taking some more Presburger functions covering the whole interval when $t$ is smaller; we will do this whenever convenient.

We will treat three special cases, namely when all $b_i$ are negative, when all $b_i$ are positive, and when all $b_i$ are zero.
This will then be put together to obtain the result for arbitrary $b_i$. For the latter to work, we will prove slightly stronger statements in the special cases, namely the following three claims.

\medskip

Claim $-1$: If all $b_i$ are negative, then there exists $\ell \in \NN$ such that for every $c = (c_i)_i$ and every $q \ge 2$, there exists $w_{-1} \in \{0, \dots, \ell-1\}$ such that
\begin{equation}
\tag{$*_{-1}$}
|h(w)| \le q^{(\ell-w)/2} \cdot |h(w_{-1})| \qquad \text{for every } w \ge 0
\end{equation}

\medskip

Claim $0$: If all $b_i$ are zero, then there exist $\ell,m \in \NN$ such that for every $c = (c_i)_i$, every $q\ge 2$ and every sufficiently big $t$, there exists $w_0 \in \{s, 2s, \dots, (\ell - 1)s\}$, where $s = \lfloor \frac{t}{\ell} \rfloor$ such that
\begin{equation}
\tag{$*_{0}$}
|h(w)| \le m \cdot |h(w_0)| \qquad \text{for every } 0 \le w \le t
.
\end{equation}

\medskip

The third claim follows from Claim $-1$ by replacing $h(w)$ by $h(t-w)$:

Claim $1$: If all $b_i$ are positive, then there exists $\ell \in \NN$ such that for every $c = (c_i)_i$, every $t$ and every $q \ge 2$, there exists $w_{1} \in \{t-\ell+1, \dots, t\}$ such that
\begin{equation}
\tag{$*_{1}$}
|h(w)| \le q^{(\ell-t+w)/2} \cdot |h(w_{1})| \qquad \text{for every } w \le t
.
\end{equation}

\medskip

Before we prove those claims, here is how they imply Lemma~\ref{lem:Z-bd}.

Write $h(w)$ as a sum $h_{-1}(w) + h_{0}(w) + h_{1}(w)$, according to the sign of the $b_i$, and
apply the corresponding claims to $h_{-1}$, $h_0$ and $h_1$.
All the possible values of $w_{-1}, w_0, w_1$ appearing in the three claims are Presburger functions in $t$; these are the functions we use to conclude the lemma, so we need
to prove that for any tuple $c$, there exists a $\nu \in \{-1, 0, 1\}$ such that $\max_w |h(w)| \le m'\cdot |h(w_\nu)|$
for some constant $m'$. The idea for this is that using that the bound on $h_{-1}$ is decreasing and the one on $h_{1}$ is increasing, we can deduce that for one of the $\nu = -1,0,1$, $|h_\nu(w_\nu)|$ dominates the other two summands of $h$ at $w_\nu$; we then can bound
$\max_{w} |h(w)|$ in terms of $h(w_\nu)$. Here are the details:

We may assume that the values $\ell$ obtained from Claims~$-1$ and 1 are the same.
By imposing a lower bound on $t$, we can ensure that
\begin{equation}\label{eq:t-big}
\ell + 4 \le  w_0 \le t - \ell - 4.
\end{equation}
Indeed, we have $w_0 \ge s > t/\ell - 1$, so imposing $t \ge \ell\cdot(\ell + 5)$ implies the left hand inequality;
the right hand inequality is obtained in a similar way using $w_0 \le (\ell - 1)s \le t - t/\ell$.

It is sufficient to bound $\max_{\ell \le w \le t-\ell} |h(w)|$, since all $w$ outside of this interval are equal to one of the definable functions anyway.

Let $\nu \in \{-1, 0, 1\}$ be such that
$B_\nu$ is maximal among
\[
B_{-1} := |h_{-1}(w_{-1})|,\quad
B_0 := 3m\cdot |h_{0}(w_{0})|,\quad
B_1 := h_{1}(w_{1}).
\]
We will prove that for this choice of $\nu$, we have
\begin{equation}\label{eq:h-hnu}
h_{\nu}(w_{\nu}) \le 3 h(w_{\nu}).
\end{equation}
This then implies, for $\ell \le w \le t-\ell$:
\begin{align*}
|h(w)| \le\,\,& |h_{-1}(w)| + |h_{0}(w)| + |h_{1}(w)|
\\
\overset{\hskip-6ex(*_{-1}), (*_{0}), (*_{1})\hskip-6ex}{\le}\,\,&
\hskip4ex |h_{-1}(w_{-1})| + m|h_{0}(w_{0})| +
|h_{1}(w_{1})|
\\
\le\,\,& m' \cdot |h_\nu(w_\nu)|
 \le 3m'\cdot |h(w_\nu)|
\end{align*}
for some suitable multiple $m'$ of $m$ ($m' = 7m$ works independently of $\nu$);
this implies the lemma. Let us now prove
(\ref{eq:h-hnu}).

Suppose first that $\nu = -1$.
Then we have
\begin{equation}\label{eq.ineq}
\begin{aligned}
|h_0(w_{-1})| &\overset{(*_{0})}{\le}  m\cdot |h_{0}(w_{0})|
  \le \textstyle{\frac13} |h_{-1}(w_{-1})| \qquad\text{and}
\\
|h_1(w_{-1})| &\overset{(*_{1})}{\le} q^{(\ell-t+w_{-1})/2}\cdot |h_{-1}(w_{-1})| \le \textstyle{\frac13}
|h_{-1}(w_{-1})|,
\end{aligned}
\end{equation}
where the last inequality follows from $t - \ell \overset{(\ref{eq:t-big})}{\ge} \ell + 8 \ge w_{-1} + 9$.
This yields
\begin{equation}\label{eq.inecomp}
\begin{aligned}
|h(w_{-1})| &= |h_{-1}(w_{-1}) + h_{0}(w_{-1}) + h_{1}(w_{-1})|\\
&\ge |h_{-1}(w_{-1})| - |h_{0}(w_{-1})| - |h_{1}(w_{-1})|\\
&\overset{(\ref{eq.ineq})}{\ge} |h_{-1}(w_{-1})|\cdot (1 - \frac13 - \frac13)
\end{aligned}
\end{equation}
and hence $|h_{-1}(w_{-1})| \le 3|h(w_{-1})|$.

The case $\nu = 1$ works in exactly the same way.

Finally, suppose $\nu = 0$. Then we get
\[
|h_{-1}(w_0)| \overset{(*_{-1})}{\le}
q^{(\ell-w_0)/2} \cdot |h(w_{-1})|
\le \textstyle{\frac13}|h(w_{-1})|,
\]
using $w_0 \ge \ell + 4$. Analogously, we get
$|h_{1}(w_0)| \le \frac13|h(w_{-1})|$,
and hence, by the same computation as in (\ref{eq.inecomp}),
$|h_{-1}(w_{0})| \le 3|h(w_{0})|$.
This finishes the proof that the three claims imply Lemma~\ref{lem:Z-bd}.

\medskip

Before proving the claims, we carry out similar (but simpler) arguments for Lemma~\ref{lem:Z-ubd}.
Again, we write
$h(w)$ as a sum $h_{-1}(w) + h_{0}(w) + h_{1}(w)$, according to the sign of the $b_i$. Since the conclusion of Lemma~\ref{lem:Z-ubd} only speaks about tuples $c$ for which $h$ is bounded,
we may assume that $h_0$ is constant and $h_1$ vanishes entirely.

We apply Claim~$-1$ to $h_{-1}$, we don't need Claim~$1$,
and instead of using Claim~0, we simply use $w_0 := \ell + 4$ (for the $\ell$ from Claim~$-1$). Then the same arguments as for Lemma~\ref{lem:Z-bd} yield
\[
\max_{w} |h(w)| \le m'\cdot \max_{w \le \ell + 1} |h(w)|.
\]
(This time, we do a case distinction on which of
$B_{-1} := |h_{-1}(w_{-1})|$ and $B_0 := 2\cdot |h_{0}(w_{0})|$ is bigger.)

Thus for both lemmas, it remains to prove the three claims. More precisely, it suffices to prove Claims 0 and $-1$, since Claim 1 is equivalent to Claim $-1$.

\medskip

\emph{Proof of Claim 0}.
The function $h(w) =  \sum_i c_i w^{a_i}$ is a polynomial of degree $d := \max_i a_i$.
Let $V$ be the vector space of all polynomials of degree $d$ and consider the map
\[
\phi\colon V \to \CC^{d+1},
f \mapsto \big(f(\frac{1}{2d+4}), f(\frac{2}{2d+4}), \dots, f(\frac{d+1}{2d+4})\big).
\]
Set $S := \{f \in V \mid \|\phi(f)\|_\infty = 1\}$,
i.e., the preimage of the unit sphere in $\CC^{d+1}$ with respect to the maximum norm.
Since $\phi$ is injective, $S$ is compact, so the maximum
\[
m := \max\{|f(\lambda)| \mid f \in S, 0 \le \lambda \le 1\}
\]
exists. From this, we deduce that Claim 0 holds for $\ell = d + 2$.
Indeed, set $f(\lambda) := h(\alpha\lambda)$, where
\[
\alpha := (2d+4)\cdot s = 2(d+2) \lfloor \frac{t}{d+2}\rfloor \ge 2(d+2) (\frac{t}{d+2} - 1).
\]
By imposing $t \ge 2(d+2)$, we obtain $\alpha \ge t$, and hence $\alpha^{-1}w \in [0,1]$
for $0 \le w \le t$. Thus
\[
|h(w)| = |f(\alpha^{-1} w)| \le m\cdot \|\phi(f)\|_\infty
= m\cdot \max_{j=1, \dots, d+1} |h(js)|,
\]
which is what we had to show.

\medskip

\emph{Proof of Claim $-1$}.
Set $A := \max_i a_i$, $B := \max_i -b_i$ and $I = \{0, 1, \dots, A\} \times \{-1, -2, \dots, -B\}$. We can write $h(w)$ as
\[
h(w) = \sum_{(a,b) \in I} c_{a,b}w^a q^{bw}.
\]
Let $c = (c_{a,b})_{(a,b) \in I}$ be the tuple of all coefficients.

Set $\ell := (A+1)\cdot B$ and $x := (h(0), \dots, h(\ell-1)) \in \CC^\ell$.
We will find an $N$ depending only on $A$ and $B$ (but not on $c$)
such that
\begin{equation}\label{eq:c1-xinfty}
\|c\|_1 \le q^N\cdot \|x\|_\infty,
\end{equation}
where $\|\cdot\|_1$ and $\|\cdot\|_\infty$ are the usual norms on $\CC^\ell$. This then implies, for all $w \ge 0$,
\[
\begin{array}{rc@{}c@{}c}
|h(w)| \le& \|c\|_1 &\,\,\cdot\,\,& \max_{(a,b) \in I} w^a q^{bw}\\
\le& q^N \max_{w' < \ell} |h(w')| &\,\,\cdot\,\,& w^A q^{-w}\\
\le&
\multicolumn{3}{l}{q^{(\ell' - w)/2} \cdot \max_{w' < \ell'} |h(w')|,}
\end{array}
\]
where the last inequality is ensured by choosing $\ell'$ big enough, namely such
that $N + A \log_q w \le \ell'/2 + w/2$ holds for all $w \ge 0$ and $q \ge 2$.
Thus (\ref{eq:c1-xinfty}) implies Claim~$-1$.

To obtain (\ref{eq:c1-xinfty}), instead of bounding $\|c\|_1$, we may as well bound $\|c\|_\infty$ (since they
differ at most by a factor $\ell$). We have $x = Zc$, where $Z = (z_{w,(a,b)})_{0\le w<\ell, (a,b) \in I}$
is the matrix with coefficients
\[
z_{w,(a,b)} = w^aq^{bw}.
\]
Suppose for the moment that $Z$ is invertible. Then, by definition of the operator norm $\|Z^{-1}\|_\infty$,
we have $\|c\|_\infty \le \|Z^{-1}\|_\infty \cdot \|x\|_\infty$, so it suffices to find an
upper bound on $\|Z^{-1}\|_\infty$ of the form $q^N$,
where $N$ only depends on $A$ and $B$.
(Note that the entire matrix $Z$ only depends on $A$, $B$ and $q$.)

Up to a constant factor (namely $\ell$),
$\|Z^{-1}\|_\infty$ is bounded by the maximum of the absolute values of the entries of $Z^{-1}$.
These entries are of the form $\det(Z')/\det(Z)$, where $Z'$ is a minor of $Z$. Since
$\det(Z')$ is a polynomial in the entries of $Z$ and those entries are bounded (even independently of $q$,
since all exponents $bw$ are non-positive),
we have an upper bound on $|\det(Z')|$, and it remains to find a lower bound on $|\det(Z)|$ of the form $q^{-N}$;
this at the same time will prove that $Z$ is invertible.

Considering $q$ as an indeterminate, we have $\det(Z) \in \ZZ[q^{-1}] \subseteq \ZZ[q, q^{-1}]$.
It is enough to prove that $\det(Z) \ne 0$ when considered as such a Laurent polynomial.
Indeed, this then implies that for $q$ big enough, we have $|\det(Z)| > \alpha\cdot q^{-N}$
for some fixed $\alpha > 0$, where $-N$ is the least negative power of $q$ appearing in $\det(Z)$.

We have
\begin{equation}\label{eq:detZ}
\det(Z) = \sum_{\sigma} \sgn(\sigma)\prod_{(a,b) \in I} z_{\sigma(a,b),(a,b)}
\end{equation}
where $\sigma$ runs over all bijections $I \to \{0, \dots, \ell - 1\}$.
(We fix an order on $I$ for $\sgn(\sigma)$ and the sign of $\det(Z)$ to be well-defined.)
Each summand of the sum (\ref{eq:detZ}) is a monomial in $q$, namely of degree
\[
d(\sigma) := \sum_{(a,b) \in I}b\sigma(a,b) = \sum_{-B \le b \le -1} b\cdot\sum_{0 \le a \le A} \sigma(a,b).
\]
Let $d_0$ be the smallest (i.e., most negative) degree in $q$ appearing among the summands in (\ref{eq:detZ}).
To prove $\det(Z) \ne 0$,
we will prove that the sum $R$ of the monomials of degree $d_0$ in $q$ is non-zero.

Let us write $\{0, \dots, \ell - 1\}$ as a disjoint union of $B$ many intervals $J_b$ of length $A$:
\[
J_b := \{n \in \ZZ \mid (A+1)(-1-b) \le n \le (A+1)(-1-b) + A\},
\]
for $b = -1, \dots, -B$.
The degree $d(\sigma)$ is minimal if and only if,
for every $(a,b), (a', b') \in I$ with $b < b'$, we have $\sigma(a,b) > \sigma(a', b')$.
This is equivalent to $\sigma(\cdot, b)$ sending $\{0, \dots, A\}$ to $J_b$ for every $b$.

Write $S_b$ for the set of bijections $\{0, \dots, A\} \to J_b$.
Then the sum of the monomials of minimal degree in (\ref{eq:detZ}) is (maybe up to sign)
\begin{align*}
R &= \sum_{\sigma_{-1} \in S_{-1}} \dots \sum_{\sigma_{-B} \in S_{-B}} \sgn(\sigma_{-1})\cdots\sgn(\sigma_{-B})
\prod_{(a,b) \in I} z_{\sigma_b(a),(a,b)}\\
&= \sum_{\sigma_{-1} \in S_{-1}} \dots \sum_{\sigma_{-B} \in S_{-B}}
\prod_b
\bigg(\sgn(\sigma_{b})
\prod_{a} z_{\sigma_b(a),(a,b)}\bigg)\\
&=
\prod_b\sum_{\sigma_{b} \in S_{b}}
\bigg(\sgn(\sigma_{b})
\prod_a z_{\sigma_b(a),(a,b)}\bigg)
\\
&= q^{d_0} \prod_b
\sum_{\sigma_b \in S_b}\bigg( \sgn(\sigma_b)\prod_a \sigma_b(a)^a\bigg).
\end{align*}
Each factor in the product over $b$ is a Vandermonde determinant
(corresponding to a polynomial of degree $A$ evaluated at each element of $J_b$) and hence non-zero.
Thus $R \ne 0$, which is what remained to be proven.
\end{proof}

\begin{proof}[Proof of Theorem \ref{thm:fam:gen}]
By Corollary \ref{from.exp.to.e.IML}, it is enough to consider
a non-negative real-valued $H \in \cCexp(W \times X)$ for $W\subset \RFss\times \VG^*$ and find a
non-negative real-valued $G \in \cCexp(X)$ and a definable $d \in \VG$ such that
\begin{equation}\label{eq:sq:prf}
\sup_{w \in W} H(w, x) \le G(x) \le q^{d}  \sup_{w \in W} H(w, x).
\end{equation}
It will be handy to consider a slightly more general situation, namely where $H$ lives on a definable
subset $U \subset W \times X$ and where the suprema run over $w \in U_x$.

By a recursive procedure, it is enough to only consider the situations where $W\subset \RF_n$ for some $n$ or $W  \subset \VG$.
In the first case, we set
$$
G(x) := \sum_{w\in U_x} H(w,x),
$$
(which lies in $\cCexp(X)$ by Theorem~\ref{thm:mot.int.}).
This $G(x)$ is obviously at least as big as $\sup_{w \in U_x}H(w,x)$, and it exceeds the supremum
at most by a factor $\#\RF_n = q^{\ord(n) + 1}$, so it is as desired.

In the case $W \subset \VG$, the idea is to use a rectilinearization result as in \cite{CPres} to reduce to
Lemmas \ref{lem:Z-ubd} and \ref{lem:Z-bd}, though to deal with our setting allowing highly ramified fields,
we need rectilinearization in the form stated in \cite[Proposition~5.2.6]{CHallp}.
The details are as follows.

By \cite[Proposition~5.2.3]{CHallp}, after possibly introducing new residue ring variables
(which we can later get rid of again as explained above), and after a finite partition of $U$,
we may suppose that ``$H$ has linear ingredients'', i.e., all functions $U \to \VG$ involved in the definition of $H$ depend linearly on $w$.
Moreover, by \cite[Proposition~5.2.6]{CHallp}
we may suppose that either $U_x = \NN$ for all $x$ or $U_x = \{w \in \VG \mid 0 \le w \le \alpha(x)\}$
for some definable function $\alpha:X\to \VG$; this involves introducing more new residue ring variables,
another finite partition of $U$, and applying a bijection which is linear over $X$.
(The linearity of this bijection ensures that $H$ still has linear ingredients.)

%

That $H$ has linear ingredients means that there exist $k$ and $a_i,b_i$ ($1 \le i \le k$) such that
for any $F$, $\psi$ and $x$, the map $w\mapsto H_{F,\psi}(w,x)$ is of the form
\[
h(w) := \sum_i c_i w^{a_i} q^{b_i w}
\]
for some $c_i$ depending on $F$, $\psi$ and $x$. Thus we can either apply
Lemma \ref{lem:Z-ubd} (if $U_x = \NN$) or Lemma \ref{lem:Z-bd} with $t = \alpha_F(x)$. In
both cases, we obtain an integer $m \ge 1$ and finitely many definable functions $w_1, \dots, w_\ell\colon X \to \VG$
with $w_i(x) \in U_x$ such that
\begin{equation}\label{lemZresult}
\sup_{w \in U_x}H(w,x) \le m\cdot\max_{1\le i\le \ell} H(w_i(x),x)
\end{equation}
for all $x$. (In the $U_x = \NN$ case, we take $w_i$ to be constant equal to $i - 1$.) Now set
$$
G(x) := m \sum_{i=1}^\ell H(w_i(x),x).
$$
Then we have
\[
\sup_{w \in U_x} H(w, x) \overset{(\ref{lemZresult})}{\le} G(x) \le m\cdot \ell\cdot \sup_{w \in W} H(w, x).
\]
So this $G$ is as desired, given that $m\cdot \ell$ can easily be bounded by an integer power of $q$.
\end{proof}

\section{Limits of $\cCexp$-functions}
\label{sec:limits}

This section contains various results about limits, namely: Existence of limits in various contexts is given by
$\cCexp$-conditions, and the limit itself is a $\cCexp$-function in the parameters.
We consider both, classical pointwise limits, and  $L^p$-limits. Even though we obtain
a whole variety of different results, it should be noted that certain other, subtly different
statements do not seem to hold; see the remark after Theorem~\ref{limits}.

\subsection{Pointwise limits and continuity}
The first result states that limits being $0$ is a $\cCexp$-condition and that convergence cannot
be arbitrarily slow.

\begin{thm}[$0$-limits]\label{thm:limits0}
Let $H$ be in $\cCexp(X \times \VG^n)$,
where $X$ is a definable set and where $n\geq 0$.
\begin{enumerate}
 \item
 The set $\{x \in X \mid \lim\limits_{\lambda,\ \|\lambda\|\to +\infty} |H (x,\lambda) |_{\C} = 0\}$ is a $\cCexp$-locus.
 \item
There exists a rational number $r<0$ and a definable function $\alpha: X\to \VG$, such that for each $F\in \Loc_{ \gg1}$, for each $\psi \in \cD_F$ and each $x\in X_F$, the following holds.

If one has
\begin{equation}\label{limits:upper000}
\lim\limits_{\lambda,\ \|\lambda\|\to +\infty} |H_{F,\psi} (x,\lambda) |_{\C}  =0 ,
\end{equation}
then one actually has
\begin{equation*}
| H_{F,\psi} (x,\lambda) |_{\C}  \le q_F^{ r \|\lambda\|} \mbox{ for all  $\lambda \in \Z^n$ with $\| \lambda \|>\alpha_F(x)$}.
\end{equation*}
\end{enumerate}
\end{thm}

The above result also holds for more general kinds of limits:

\begin{thm}[$0$-limits]\label{thm:limits0:gen}
Let $H$ be in  $\cCexp(X \times Y)$,
where $X$ and $Y$ are definable sets, and let $\gamma:Y\to \VG_{\ge 0}$ be a surjective definable function.
\begin{enumerate}
 \item
 The set $\{x \in X \mid \lim\limits_{y\in Y,\ \gamma(y)\to +\infty} |H (x,y) |_{\C} = 0\}$ is a $\cCexp$-locus.
 \item
There exists a rational number $r<0$ and a definable function $\alpha: X\to \VG$ such that for all $F\in \Loc_{ \gg 1}$, for each $\psi \in \cD_F$ and each $x\in X_F$, the following holds.

If one has
\begin{equation}\label{limits:upper0}
\lim\limits_{y\in Y_F,\ \gamma_F(y)\to +\infty} |H_{F,\psi} (x,y) |_{\C}  =0 ,
\end{equation}
then one actually has
\begin{equation}\label{limits:upper}
| H_{F,\psi} (x,y) |_{\C}  \le q_F^{ r \gamma_F(y) } \mbox{ for all  $y\in Y_F$ with $\gamma_F(y)>\alpha_F(x)$}.
\end{equation}
\end{enumerate}
\end{thm}

For some $F, \psi, x \in X_F$, the limits appearing in Theorems~\ref{thm:limits0} and \ref{thm:limits0:gen} might not even exist.
(The theorems do not assume that they do.)
The next result states that existence of limits is also a $\cCexp$-condition and that
if a limit exists, then its value is given by a $\cCexp$-function.

\begin{thm}[Limits]\label{limits}
Let $H$ be in $\cCexp(X\times \VG^n)$ for some definable set $X$.
\begin{enumerate}
 \item
  The set $\{x \in X \mid \lim\limits_{\|\lambda\|\to +\infty} H (x,\lambda) \text{ exists in $\C$}\}$ is a $\cCexp$-locus.
\item
There exists $G$ in $\cCexp(X)$ such that the following holds for all $F\in \Loc_{ \gg1}$, for each $\psi \in \cD_F$ and each $x\in X_F$.
$$
\mbox{If }
\lim_{\|\lambda\|\to+\infty}  H_{F,\psi} (x,\lambda) \mbox{ exists in $\CC$},
$$
then
$$
G_{F,\psi}(x) = \lim_{\|\lambda\|\to+\infty}  H_{F,\psi} (x, \lambda).
$$
\end{enumerate}
\end{thm}

Again, we can ask for a version of this for more general limits. We obtain:

\begin{thm}[Limits]\label{limits:gen}
Let $H$ be in $\cCexp(X \times Y)$,
where $X$ and $Y$ are definable sets, and let $\gamma:Y\to \VG_{\ge 0}$ be a surjective definable function. Then
the set $\{x \in X \mid \lim\limits_{y\in Y,\ \gamma(y)\to +\infty} H (x,y) \text{ exists in }\C\}$ is a $\cCexp$-locus.
\end{thm}

Note that surprisingly, in this generality it is not straightforward to prove that the value of the limit is of $\cCexp$-class,
and we even doubt that this is true. It seems that such a result would require some better understanding of the number of rational
points of families of varieties over finite fields. On the other hand, it should be possible to obtain the generalization
in a slightly altered context, namely after adding function symbols for Skolem functions to the sorts $\RFss$, or, after adding more denominators like in the rational motivic functions of \cite{Kien:rational}.

From Theorem~\ref{limits:gen}, one easily deduces that continuity is a $\cCexp$-condition.
In the following, given a definable set $Z \subset \VF^n$, we write $\bar Z$ for its topological closure
(i.e., $\bar{Z}_F$ is the topological closure of $Z_F$ in $F^n$ for each $F$), which is also definable.
\begin{cor}[Continuity]\label{cor.cont}
Let $H$ be in $\cCexp(X \times Y)$, where $X$ and $Y$ are definable sets and $Y \subset \VF^n$.
Then
\[\{(x,y) \in X \times Y  \mid H(x, \cdot) \text{ is continuous at } y\}\]
and
\begin{align*}
\{(x,y) \in X \times (\bar{Y}\setminus Y)  \mid \,&H(x, \cdot) \text{ has an extension to $Y\cup \{y\}$,}\\
&\text{continuous at $y$}  \}
\end{align*}
are $\cCexp$-loci.
\end{cor}
\begin{proof}
The conditions at $(x,y)$ can be expressed respectively as follows:
\[
\begin{cases}
 \lim_{y' \in Y, \ord(y' - y) \to +\infty} (H(x, y') - H(x,y)) = 0; \\
 \lim_{y' \in Y, \ord(y' - y) \to +\infty} H(x, y') \text{ exists}.
\end{cases}
\]
Now use Theorem~\ref{thm:limits0:gen} and Proposition~\ref{locbasicop}.
\end{proof}

Note that if $H \in \cCexp(Y)$ is continuous on $Y$ and if $\lim_{y' \in Y, y' \to y} H(x, y')$ exists for each $y\in\bar{Y}\setminus Y$, then the unique continuous extension of $H$ is not known by us to be of $\cCexp$-class, for the same reasons as explained after Theorem~\ref{limits:gen}.
However, it is of $\cCexp$-class in many slightly more restrictive cases. The most general setting in which this can be proven would probably be
very technical, so in the following, we just prove it under reasonable assumptions.

\begin{prop}[Continuous extension]\label{prop.cont.ext}
Let $H$ be in $\cCexp(X \times Y)$, where $X$ and $Y$ are definable sets and $Y \subset \VF^n$ and suppose that $\bar Y$ is clopen (i.e., $\bar Y_F \subset F^n$ is open and closed for every $F \in \Loc_{\gg 1}$).
Then $H$ can be extended to a $\cCexp$-function $\bar H$ on $X \times \bar Y$ such that for every $(x,y) \in X \times \bar Y$,
if $\lim_{y' \in Y, y' \to y} H(x, y')$ exists, then $\bar H(x,\cdot)$ is continuous at $y$.

Namely, with all indices, for every
$F \in \Loc_{\gg1}$, for every $\psi \in \cD_F$ and for every $(x,y) \in X_F \times \bar Y_F$,
if $\lim_{y' \in Y_F, y' \to y} H_{F,\psi}(x, y')$ exists, then $\bar H_{F,\psi}(x,\cdot)$ is continuous at $y$.	
\end{prop}
\begin{proof}
We can simply define
\[
\bar H(x,y) := \lim_{\lambda \to \infty} q^{-n\lambda} \int_{B_\lambda(y) \cap Y} H(x,y') dy',
\]
where $B_\lambda(y) \subset \VF^n$ is the closed ball of valuative radius $\lambda$ around $y$ (considered as a definable set).
More precisely, we use Theorems~\ref{thm:mot.int.} and \ref{limits} to find an $\bar H$ such that
for every $F, \psi$ and every $x \in X_F, y \in Y_F$, we have
\[
\bar H_{F,\psi}(x,y) = \lim_{\lambda \to \infty} q_F^{-n\lambda} \int_{B_{F,\lambda}(y) \cap Y_F} H_{F,\psi}(x,y') dy',
\]
whenever the integral and the limit exist, where $B_{F,\lambda}(y)$ is the concrete closed ball in $F^n$ given by $\lambda \in \ZZ$ and $y \in F^n$.

Since $\bar Y_F$ is clopen, for sufficiently
big $\lambda$, we have  $B_{F,\lambda}(y) \subset  \bar Y_F$, so that we are just averaging over this ball (where $\bar Y_F\setminus Y_F$ has measure zero by dimension considerations).
If $H_{F,\psi}(x,\cdot)$ is continuous at $y$, then the integral exists for large $\lambda$, and continuity also implies that the limit of those average values exists
and is equal to $\lim_{y' \in Y_F, y' \to y} H_{F,\psi}(x, y')$, as desired.
\end{proof}

As usual (namely, by transfer for $\cCexp$-conditions) the above theorems imply corresponding transfer principles for limits, as follows.

\begin{cor}[Transfer principle for limits and continuity]\label{trans:lim}
Let $H$ be in $\cCexp(X \times Y)$,
where $X$ and $Y$ are definable sets, and let $\gamma:Y\to \VG$ be a definable function.

Then there exists $M$ such that, for any $F\in \Loc$ of residue characteristic $\ge M$,
the truth of each of the following statements depends only on (the isomorphism class of) the residue field of $F$;
here, in the 3rd and 4th statement, we assume $Y \subset \VF^n$ for some $n$.
\[
\lim\limits_{ y\in Y_F,\ \gamma_F(y) \to +\infty} H_{F,\psi} (x,y)   =0 \text{ for all $x\in X_F$ and all $\psi\in\cD_F$};
\]
\[
\lim\limits_{ y\in Y_F,\ \gamma_F(y) \to +\infty} H_{F,\psi} (x,y)  \text{ exists for all $x\in X_F$ and all $\psi\in\cD_F$};
\]
\[
H_{F,\psi}(x,\cdot) \text{ is continuous on $Y_F$ for all $x \in X_F$};
\]
\[
H_{F,\psi}(x,\cdot) \text{ has a continuous extension to ${\bar Y}_F$  for all $x \in X_F$.}
\]
\end{cor}

The proofs of Theorems \ref{thm:limits0:gen} and \ref{limits} use Theorem \ref{thm:fam:gen} in a crucial way, to reduce to the case
of a limit over a single $\VG$-variable; other proofs, relying on Proposition \ref{IML} above and the rectilinearization result given by Proposition~5.2.6 of \cite{CHallp}, may also be thinkable. Here is the single $\VG$-variable version of Theorems \ref{thm:limits0:gen} and \ref{limits} (formulated in a more concise way):

\begin{thm}[Limits]\label{limits:basic}
Let $H$ be in $\cCexp(X \times \VG_{\ge0})$, where $X$ is a definable set. Then we have the following.
\begin{enumerate}
 \item The condition ``$\lim_{\lambda \to +\infty} H(x,\lambda)$ exists'' is a $\cCexp$-condition on $x$.
 \item There exists a function $G \in \cCexp(X)$, a definable function $\alpha\colon X \to \VG_{\ge0}$ and a rational number $r < 0$
    such that for every $F \in \Loc$, every $\psi \in \cD_F$ and every $x \in X_F$, if the limit $\lim_{\lambda \to +\infty} H_{F,\psi}(x,\lambda)$ exists, then
\begin{equation}\label{upper-basic}
|H_{F,\psi}(x,\lambda) - G_{F,\psi}(x)|_\C  \le q_F^{ r \cdot\lambda }
\end{equation}
for every $\lambda \ge \alpha_F(x)$. In particular, the limit is equal to $G_{F,\psi}(x)$.
\end{enumerate}
\end{thm}

This theorem implies Theorems~\ref{thm:limits0} -- \ref{limits:gen}  as follows.

\begin{proof}[Proof of Theorems \ref{thm:limits0} and \ref{thm:limits0:gen}]
Theorem \ref{thm:limits0} follows from \ref{thm:limits0:gen} (using $Y = \VG^n$ and $\gamma(\lambda) = \|\lambda\|$), so we now prove
Theorems \ref{thm:limits0:gen}.

We may suppose that $Y=Y_0\times \VG_{\ge0}$ and that $\gamma$ is the coordinate projection to the $\VG$-variable.
By Theorem~\ref{thm:fam:gen}, there exists a real-valued function $H'$ in $\cCexp(X \times \VG_{\ge0})$
such that $H'(x,\lambda)$ differs from $\sup_{y_0\in Y_0} |H(x,y_0,\lambda)|^2_\C$ by at most a factor
of the form $q^{d}$, for some definable $d \in \VG$.
This implies that we can without loss replace $H$ by $H'$, considering the limit $\lim_{\lambda \to +\infty} H'(x,\lambda)$;
in this version, the Theorem follows pretty directly from Theorem~\ref{limits:basic}.
Indeed, that the limit is $0$ can be expressed by the condition ``$\lim_{\lambda \to +\infty} H'(x,\lambda)$ exists
and $G(x) = 0$'', which is of $\cCexp$-class, where $G$ is the function obtained by applying Theorem~\ref{limits:basic} to $H'$,
and Claim~(\ref{limits:upper}) follows from (\ref{upper-basic}).
\end{proof}

\begin{proof}[Proof of Theorems \ref{limits} and \ref{limits:gen}]
Define $H'$ in $\cCexp(X \times \VG_{\ge0} \times Y \times Y)$ by
\[
H'(x, \lambda, y_1, y_2) =
\begin{cases}
 H(x, y_1) - H(x, y_2) & \text{if } \gamma(y_1), \gamma(y_2) \ge \lambda\\
 0  & \text{otherwise}.
\end{cases}
\]
By applying Theorem~\ref{thm:fam:gen} to $H'$, we find $G'$ in $\cCexp(X \times \VG_{\ge0})$
such that \change{
}
this $G'(x,\lambda)$ differs from $\sup_{y_1, y_2} |H(x, y_1) - H(x, y_2)|_\C^2$ by at most a factor
not depending on $\lambda$, where $y_1$ and $y_2$ both run over $\{y \in Y \mid \gamma(y) \ge \lambda\}$.
In particular we have that $\lim_{\lambda \to +\infty}G'(x,\lambda) = 0$ iff $H(x, \cdot)$
is a Cauchy sequence in the sense of the $\gamma$-limit. Thus the limit $\lim\limits_{y\in Y,\ \gamma(y)\to +\infty} H (x,y)$
exists if and only if $\lim_{\lambda \to +\infty}G'(x,\lambda) = 0$, which is a $\cCexp$-condition by Theorem~\ref{limits:basic}.
This finishes the proof of Theorem~\ref{limits:gen} and of Theorem~\ref{limits} (1).

To obtain a function $G$ as desired in Theorem~\ref{limits} (2),
we can simply take the limit of $H(x, \lambda)$ along any sequence of $\lambda$ with $\|\lambda\| \to \infty$,
so we obtain it e.g.\ by applying Theorem~\ref{limits:basic} (2) to
$$
\lim_{\mu \to +\infty} H(x, \mu, 0, \dots, 0),
$$
where $\mu$ runs over $\VG_{\ge 0}$.
\end{proof}

\begin{proof}[Proof of Theorem~\ref{limits:basic}]
We apply Proposition~\ref{repar} to $H$, yielding (using the same notation as there)
a reparameterization $\sigma\colon U = X \times \VG_{\ge0} \to U_{\rm par} \subset \RFss \times X \times \VG_{\ge0}$
and a partition of $U_{\rm par}$ into sets
\begin{equation}\label{lb.A}
A_j = \{(z,\mu) \in B_j \times \VG \mid \alpha_{j}(z) \leq \mu \le  \beta_{j}(z) \wedge \mu \equiv e_j\bmod n_j \}
\end{equation}
(for some definable $B_j \subset \RFss \times X$, $\alpha_j, \beta_j\colon B_j \to \VG \cup \{\pm \infty\}$ and some integers $e_j \ge 0, n_j > 0$) such that for $(z, \mu) \in A_j$, we have
\begin{equation}\label{lb.H}
H(\sigma^{-1}(z,\mu)) =
H_{{\rm par}} (z,\mu) = \sum_{i} c_{ij}(z) \mu^{a_{ij}} q^{b_{ij}\mu}
\end{equation}
for some $a_{ij} \in \NN$, $b_{ij} \in \QQ$ and some $c_{ij} \in \cCexp(B_{ij})$.

Denote the preimage of $A_j$ under $\sigma$ by
\[
A'_j := \sigma^{-1}(A_j) \subset U.
\]
For any fixed $F$ and $x \in X_F$, the fibers $A'_{j,F,x} = \{\lambda \in \NN \mid (x,\lambda) \in A'_{j,F}\}$ form a partition of $\NN$.
The limit $\lim_{\lambda \to +\infty}H_{F,\psi}(x,\lambda)$ exists if and only if, for each $j$ for which $A'_{j,F,x}$ is unbounded,
the corresponding restricted limit $\lim_{\lambda \to +\infty, \lambda \in A'_{j,F,x}}H_{F,\psi}(x,\lambda)$ exists and moreover all those limits are equal.
In particular, those $A_j$ for which $\beta_j \ne +\infty$ are irrelevant. Indeed, recall that for each $j$, either $\beta_j$ is constant $+\infty$, or, by
Proposition~\ref{repar}~(2), $A'_{j,F,x}$ is bounded for all $F,x$.
Therefore, in the remainder of the proof, we only consider those parts $A_j$ (and $A'_j$) for which $\beta_j = +\infty$.

After shrinking $U_{\rm par}$ to the union of those parts (and shrinking $U$ accordingly and restricting $H$ to $U$),
we can entirely get rid of the reparameterization, by replacing each $A_j$ by its preimage $A'_j$.
To see this, what we need to check is that we have analogues of (\ref{lb.A}) and (\ref{lb.H}) for those $A'_j$.
Those analogues can be obtained, provided that the restriction of $\sigma$ to $A'_j$
is of the form $\sigma(x, \mu) = (\tau(x), \mu)$ for some bijection $\tau \colon B'_j \subset X \to B_j$,
where $B'_j$ is the projection of $A'_j$ to $X$. Indeed, such a $\tau$ is automatically definable (using $\sigma$)
and hence (\ref{lb.A}) and (\ref{lb.H}) for $A'_j$ can be obtained by pre-composing $\alpha_j$ and $c_{ij}$ with $\tau$.

Since $\sigma$ is a reparameterization,
the fiber $(A'_j)_x = \{\mu \in \VG \mid (x, \mu) \in A'_j\}$ is equal
to a disjoint union of certain fibers $(A_j)_{\xi,x} = \{\mu \in \VG \mid (\xi,x,\mu) \in A_j\}$,
and what we need to check is that $(A'_j)_x$ is actually equal to a single fiber $(A_j)_{\tau(x)}$.
This follows from the fact that different $(A_j)_{\xi,x}$, $(A_j)_{\xi',x}$ cannot be disjoint, by
(\ref{lb.A}): both fibers contain all big $\mu \in \VG$ satisfying $\mu \equiv e_j \bmod n_j$.
Thus the reparameterization is indeed unnecessary, so from
now on we assume that
\[
A_j = \{(x,\mu) \in B_j \times \VG \mid \alpha_{j}(x) \leq \mu \wedge \mu \equiv e_j\bmod n_j \}
\]
for some $B_j \subset X$ and that, for $(x,\mu) \in A_j$, we have
\[
H(x,\mu) =  \sum_{i} c_{ij}(x) \mu^{a_{ij}} q^{b_{ij}\mu}.
 \]

We assume without loss that $a_{0j} = b_{0j} = 0$, and we
let $I$ be the collection of those $(i,j)$ satisfying either $b_{ij}>0$, or, $b_{ij}=0$ and $a_{ij}>0$.

Since for each fixed $j$ the pairs $(a_{ij},b_{ij})$ are distinct, given $x \in X$,
the limit exists if and only if, jointly,
\begin{equation}\label{limex1}
c_{ij}(x)=0 \mbox{ for all those $(i,j) \in I$ satisfying } x\in B_{j},
\end{equation}
and
\begin{equation}\label{limex2}
c_{0j}(x) = c_{0k}(x)
\text{ for all those $j$, $k$ satisfying $x\in B_{j}$, $x\in B_{k}$.}
\end{equation}

Each of these equations is a $\cCexp$-condition on $x$, so their conjunction is also a $\cCexp$-condition (by Proposition~\ref{locbasicop}), hence finishing the proof of (1).

When the limit exists, it is equal to the following $\cCexp$-function:
$$
G(x) := c_{0j}(x)
$$
for any $j$ with $x\in B_{j}$ (e.g.\ the smallest $j$ for some linear ordering of the $j$).

It remains to find $r < 0$ and $\alpha\colon X \to \VG$. To this end, we may from now on without loss suppose that
all $b_{ij}$ are negative (by subtracting $G(x)$, and by ignoring those $x$ for which the limit does not exist).
In particular, the limit is $0$ if it exists.

We need to choose $r$ and $\alpha$ in such a way that we have
\begin{equation}\label{r-alpha}
|\sum_{i} c_{ij}(x) \lambda^{a_{ij}} q^{b_{ij}\lambda}|_\CC \le q^{r\lambda}
\quad \text{for all $\lambda \ge \alpha(x)$.}
\end{equation}

Let $a_0$ and $b_0$ be the maximum of all $a_{ij}$ and of all $b_{ij}$, respectively, and set $r := b_0/2$. Then to obtain (\ref{r-alpha}), it
suffices to prove
\[
N \cdot (\max_{ij} |c_{ij,F,\psi}(x)|_\C)\cdot \lambda^{a_0} \le q_F^{-r\lambda}
\]
for all $F, \psi, x$, where $N$ is the number of summands in (\ref{r-alpha}).

By the same argument as in the proof of Theorem~\ref{thm:fam} (or by actually applying Theorem~\ref{thm:fam} with $W$ being a singleton),
we obtain a bound of the form $|c_{ij,F,\psi}(x)|_{\C} \le q_F^{g_F(x)}$ for some definable $g\colon X \to \VG$.
From this, one can easily obtain an $\alpha$ as desired (e.g. choosing it such that $g_F(x) \le -r\alpha_F(x)/2$
and $N\lambda^{a_0} \le q_F^{-r\lambda/2}$ for $\lambda \ge \alpha_F(x)$).
\end{proof}

\subsection{Limits of functions}

In this section, we consider various types of limits of sequences of functions. For simplicity we restrict to functions
on $\VF^n$ (though the notions would also make sense on other definable sets, using the counting measures on $\VG$ and on $\RF_*$,
and the proofs also go through in this generalized setting). We ask the usual two questions: Does the limit exist and if yes, what is it?
Those results will be used to study Fourier transforms of $L^2$ functions in Section \ref{sec:Fourier}.

By uniform convergence of a sequence of functions $f_\mu$, $\mu  \in \NN$, we mean convergence with respect to the sup norm, and
by $L^p$-convergence (for $p \in \RR_{\ge1} \cup \{\infty\}$), we mean convergence with respect to the usual $L^p$-norm.
(In particular, $L^{\infty}$-convergence means uniform convergence except on a subset of measure zero.) Note that we do not require
the individual functions $f_\mu$ to be $L^p$-integrable; instead, $L^p$-convergence towards a limit function $g$ means that
the $L^p$-norm of the difference $f_\mu - g$ exists for $\mu$ sufficiently big and that it goes to $0$.

\begin{thm}[Limits of sequences of functions]\label{L^pcom}
Let $H$ be in $\cCexp(X\times \VF^n\times \VG_{\ge0})$ for some definable set $X$ and some $n \ge 0$.
Given $F\in \Loc_{ \gg1}$, $\psi \in \cD_F$ and $x\in X_F$, consider the sequence of functions (on $F^n$)
\begin{equation}\label{eq.seq}
y \mapsto H_{F,\psi}(x, y, \mu),
\end{equation}
indexed by integers $\mu\geq 0$.
\begin{enumerate}
 \item
 For each of the notions of convergence listed below, the set of $x \in X$ such that the sequence (\ref{eq.seq}) of functions converges in that sense is a $\cCexp$-locus.
 The notions of convergence are: pointwise convergence, uniform convergence, $L^2$-convergence, $L^\infty$-convergence.
 \item
There exists a $G$ in $\cCexp(X\times \VF^n)$ such that the following holds for all $F\in \Loc_{ \gg1}$, for each $\psi \in \cD_F$ and each $x\in X_F$.

If the sequence (\ref{eq.seq}) of functions converges pointwise or uniformly or in $L^p$-norm for any $p \in \RR_{\ge 1} \cup \{\infty\}$,
then the function
$$
y\mapsto G_{F,\psi}(x,y)
$$
is the limit of this sequence.
\end{enumerate}
\end{thm}

We believe that also part (1) of this theorem is true for $L^p$-convergence for any $p \in \RR_{\ge 1} \cup \{\infty\}$. However, the proof would be more involved (even in the case $p = 1$).

From the theorem, we deduce various related results (some of which would also have simple direct proofs).

\begin{cor}[$L^p$-integrability]\label{cor.finLp}
Fix $p \in \{1, 2, \infty\}$ and let $H$ be in $\cCexp(X\times \VF^n)$. Then the set of $x \in X$ for which
$H(x,\cdot)$ is $L^p$-integrable is a $\cCexp$-locus.
\end{cor}

\begin{proof}
The case $p = 1$ is exactly Theorem~\ref{thm:mot.int.} (1). (We repeat the result for $p=1$ here only for completeness.)
For $p = 2, \infty$, we obtain the corollary by applying Theorem~\ref{L^pcom} to $H'(x,y,\mu) := q^{-\mu}H(x,y)$. Indeed, for any fixed $F, \psi, x$,
the $L^p$-limit $\lim_{\mu \to \infty} H'_{F,\psi}(x, \cdot, \mu)$ exists (and is $0$) if and only if $H_{F,\psi}(x,\cdot)$ is $L^p$-integrable.
\end{proof}

\begin{cor}[The ``almost everywhere'' quantifier]\label{cor.almost}
Suppose that $P(x, y)$ is a $\cCexp$-condition, where $x$ runs over a definable set $X$ and $y$ runs over $\VF^n$.
Then
\[
P(x,y) \text{ holds for almost all } y \in \VF^n
\]
is a $\cCexp$-condition on $x$. (Here, by ``almost all'', we mean that the complement has measure $0$.)
\end{cor}

\begin{proof}
Choose $H \in \cCexp(X \times \VF^n)$ such that $H(x,y)$ is $0$ if and only if $P(x,y)$ holds,
and set $H'(x,y,\mu) := q^{\mu} H(x,y)$. Then for any $F, \psi, x, y$, the sequence $H'_{F,\psi}(x,y,\mu)$ diverges for $\mu \to \infty$ if and only
$P_{F,\psi}(x,y)$ does not hold.
Thus $P_{F,\psi}(x,y)$ holds for almost all $y$ if and only if the $L^\infty$-limit of $H'_{F,\psi}(x,\cdot,\mu)$ exists,
so that the corollary follows by applying Theorem~\ref{L^pcom} to $H'(x,y,\mu)$.
\end{proof}

As usual, we also obtain some transfer results:

\begin{cor}[Transfer principles for limits of functions]\label{trans:funlim}
Let $H$ be in $\cCexp(X \times \VF^n \times \VG_{\ge 0})$ for some definable set $X$ and some $n \ge 0$.
Given $F\in \Loc_{ \gg1}$, $\psi \in \cD_F$ and $x\in X_F$, consider the sequence of functions (on $F^n$)
\begin{equation}\label{eq.seq.trans}
y \mapsto H_{F,\psi}(x, y, \mu),
\end{equation}
indexed by integers $\mu\geq 0$.

Fix a notion of convergence among: pointwise convergence / uniform convergence / convergence in $L^2$-norm / in $L^\infty$-norm.

Then there exists $M$ such that, for any $F\in \Loc$ of residue characteristic $\ge M$,
the truth of each of the following statements depends only on (the isomorphism class of) the residue field of $F$.
\[
\text{The sequence (\ref{eq.seq.trans}) converges for all $x\in X_F$ and all $\psi\in\cD_F$};
\]
\[
\text{The sequence (\ref{eq.seq.trans}) converges to $0$ for all $x\in X_F$ and all $\psi\in\cD_F$}.
\]
\end{cor}

Concerning transfer of convergence to $0$, let $G$ be obtained from $H$ as in Theorem~\ref{L^pcom} (2). Then
convergence to $0$ can be expressed by a $\cCexp$-condition as follows (using Theorem~\ref{L^pcom} (1) and Proposition~\ref{locbasicop}):
\[
\forall x\colon \big(\text{the sequence $H(x,\cdot,\mu)$ converges}
\wedge \forall y\colon G(x, y) = 0\big)
\]

\begin{cor}[Transfer principles for $L^p$-norms]\label{trans:Lplim}
Let $H$ be in $\cCexp(X \times \VF^n)$ for some definable set $X$ and some $n \ge 0$.
Then there exists $M$ such that, for any $F\in \Loc$ of residue characteristic $\ge M$,
the truth of each of the following four statements depends only on (the isomorphism class of) the residue field of $F$.
\[
\text{$H_{F,\psi}(x, \cdot)$ has finite $L^1$-/$L^2$-/$L^\infty$-norm
for all $x\in X_F$ and all $\psi\in\cD_F$};
\]
\[
\text{$H_{F,\psi}(x, \cdot)$ is $0$ on almost all of $F^n$,
for all $x\in X_F$ and all $\psi\in\cD_F$}.
\]
\end{cor}

Let us now prove Theorem~\ref{L^pcom}.
One ingredient to the second part is that existence of the $L^p$-limit implies
almost everywhere existence of the point-wise limit. This is not true in general, but it is true for sequences of functions of $\cCexp$-class.
Here is the precise statement:

\begin{lem}\label{Lp-vs-pointwise}
Fix $p \in \R_{\ge1} \cup \{+\infty\}$ and let $H$ be in $\cCexp(X \times \VF^n\times \VG_{\ge0})$ for some definable set $X$.
Also fix $F \in \Loc_{\gg1}$, $\psi \in \cD_F$ and $x \in X_F$.
If the sequence of functions $H_{F,\psi}(x, \cdot, \lambda)$ (on $F^n$) converges in $L^p$-norm for $\lambda \to \infty$,
then $H_{F,\psi}(x, y, \lambda)$ converges for almost all $y \in F^n$.
\end{lem}

(Without the assumption of $H$ being of $\cCexp$-class, one can only find a subsequence which converges for almost all $y$.)

\begin{proof}[Proof of Lemma~\ref{Lp-vs-pointwise}]
We apply the same reasoning as in the proof of Theorem~\ref{limits:basic}, using the same notations (except that
our set $X \times \VF^n$ plays the role of the set $X$ from the proof of Theorem~\ref{limits:basic}), namely:
We find a reparameterization $\sigma\colon U = X \times \VF^n \times \VG_{\ge 0} \to U_{\rm par}$, a partition of $U_{\rm par}$
into finitely many $A_j$, we can disregard those $A_j$ that have an upper bound $\beta_j \ne +\infty$, and after that,
we can get rid of the reparameterization again in the same way as for Theorem~\ref{limits:basic}.

Fix $F$, $\psi$ and $x \in X_F$ for the entire remainder of the proof, and
suppose that there exists a set $Y \subset F^n$ of positive measure such that
$H_{F,\psi}(x, y, \cdot)$ does not converge for any $y \in Y$. For each such $y \in Y$, either (\ref{limex1}) or (\ref{limex2}) fails.
Thus, up to shrinking $Y$, either

(a) there exists an
$A_j$ with $\{x\} \times Y \subset B_j$ such that $c_{ij,F,\psi}(x,y) \not= 0$ for all $y$ in $Y$, where $(i,j) \in I$, or

(b)
there exist
$A_j$, $A_k$ with $\{x\} \times Y \subset B_j \cap B_k$ such that $c_{0j,F,\psi}(x,y) \ne c_{0k,F,\psi}(x,y)$ for all $y$ in $Y$.

If (a) holds, we fix $j$ and choose, among the different $i$ for which (a) holds, the one
corresponding to the dominant term, i.e., such that $b_{ij}$ is maximal, and among equal $b_{ij}$, the one such that $a_{ij}$
is maximal. Then for big $\lambda \equiv e_j \mod n_j$, the $L^p$-norm of $H_{F,\psi}(x, \cdot, \lambda)$ restricted to $Y$ is eventually larger than
half of the
$L^p$-norm of the dominant term $c_{ij,F,\psi}(x,y)\lambda^{a_{ij}}q_F^{b_{ij}\lambda}$ on $Y$
and hence diverges.

Now suppose that no $A_j$ as in (a) exists. Then we are in case (b), and $A_j$ and $A_k$ correspond to two different sub-sequences
of $H$, which, on $Y$, converge to two different functions, namely $c_{0j,F,\psi}(x,\cdot)$ resp.\ $c_{0k,F,\psi}(x,\cdot)$.
Again, this contradicts $L^p$-convergence of $H_{F,\psi}(x, \cdot, \cdot)$.
\end{proof}

\begin{proof}[Proof of Theorem~\ref{L^pcom}]
We first prove (2). Using Theorem~\ref{limits}, we find a function
$G \in \cCexp(X \times \VF^n)$ such that
$$
G_{F,\psi}(x,y) = \lim_{\lambda\to\infty}H_{F,\psi}(x,y,\lambda)
$$
for every $F, \psi$, $x \in X_F$ and $y \in F^n$ for which that limit exists.
In particular, if the sequence $H_{F,\psi}(x,\cdot,\lambda)$ converges pointwise (or uniformly),
then the limit is $G_{F,\psi}(x,\cdot)$. If $H_{F,\psi}(x,\cdot,\lambda)$ converges in $L^p$-norm,
then $H_{F,\psi}(x,y,\lambda)$ converges for almost all $y$ by Lemma~\ref{Lp-vs-pointwise},
and for those $y$ for which it does, the limit is $G_{F,\psi}(x,y)$. Therefore, $G_{F,\psi}(x,\cdot)$
is also the $L^p$-limit.

We now consider (1) with the various notions of convergence. By replacing $H(x,y,\lambda)$
by $H(x,y,\lambda) - G(x,y)$ (with $G$ obtained from Part (2)), we may assume that if $\lim_{\lambda\to\infty}H_{F,\psi}(x,y,\lambda)$ exists,
then it is equal to $0$.

Pointwise convergence: This can be expressed as
\[
\forall y\colon \lim_{\lambda\to\infty}H_{F,\psi}(x,y,\lambda) \text{ exists},
\]
which is a $\cCexp$-condition by Theorem~\ref{limits:basic}.

Uniform convergence: Apply Theorem~\ref{thm:limits0:gen} to $H$, where $\gamma\colon Y \times \VG_{\ge0} \to \VG$
just sends $(y,\lambda)$ to $\lambda$.

$L^2$-convergence: The function
\[
(x,\lambda) \mapsto \int_{\VF^n}|H(x,y,\lambda)|^2_\CC dy
\]
is of $\cCexp$-class
(since the square of the absolute value can be obtained by multiplying with the complex conjugate, and then using Theorem~\ref{thm:mot.int.})
so this going to 0 is a $\cCexp$-condition by Theorem~\ref{limits:basic}.

$L^\infty$-convergence:
By Theorem 4.4.3 of \cite{CHallp}
$\cCexp$-functions are locally constant almost everywhere; more precisely,
there exists a definable set $Z \subset X \times \VF^n \times \VG$ such that for every $F$, $\psi$, $x$ and $\lambda$,
the set $F^n \setminus Z_{x,\lambda}$ has dimension less than $n$ and $H_{F,\psi}(x,\cdot,\lambda)$ is locally constant on
$Z_{x,\lambda}$. Define $H' \in \cCexp(X \times \VF^n \times \VG)$ to be equal to $H$ on $Z$ and equal to $0$ outside of $Z$. Then
$H'$ and $H$ have the same behavior concerning $L^\infty$-convergence, but $H'$ has the additional property
that $L^\infty$-convergence is equivalent to uniform convergence; this has already been dealt with above.
\end{proof}

\subsection{Stability under Fourier Transformation for $L^2$-functions of $\cCexp$-class}\label{sec:Fourier}

That the Fourier transform of an $L^1$-function of $\cCexp$-class is again of $\cCexp$-class follows directly from closedness under
integration; see \cite[Corollary 4.3.1]{CHallp} for a version in the current context. For $L^2$-functions of $\cCexp$-class,
this can be obtained using our formalism as follows.
It seems that this result is new even for fixed $F$. It compares to a real counterpart about stability under $L^2$ Fourier transformation of Theorem 8.10 of \cite{CCMRS}.

\begin{thm}[Stability under $L^2$ Fourier transformation]\label{stab:four:II}
Let $H$ be in $\cCexp(X\times \VF^m)$ for some $m\geq 0$ and some definable set $X$. Then there exists $\cF_{/X}(H)$ in $\cCexp(X\times \VF^m)$ such that the following holds for all $F$ in $\Loc_{\gg 1}$, for all $\psi\in \cD_F$ and for each $x\in X_F$.
\quote{If the map $H_{F,\psi,x} : y\mapsto H_{F,\psi}(x,y)$ is $L^2$-integrable on $F^m$, then $z\mapsto \cF_{/X}(H)_{F,\psi}(x,z)$ is the Fourier transform of $H_{F,\psi,x}$. }
\end{thm}

\begin{proof}
The Fourier transform of an $L^2$-function $H_{F,\psi}(x,\cdot)$ is defined as the $L^2$-limit of the sequence of functions
\[
G_{F,\psi}(x,\cdot,\lambda)\colon z \mapsto \int_{y\in F^m, \ord_F(y) > -\lambda} H_{F,\psi} (x,y) \psi ( y\cdot z )   |dy|,
\]
indexed by $\lambda \in \NN$ (where $z$ runs over $F^m$ and $y \cdot z$ is the scalar product).
By Theorem \ref{thm:mot.int.} (and the comment below that theorem), $G$ is of $\cCexp$-class, so the limit is of $\cCexp$-class by Theorem~\ref{L^pcom}.
\end{proof}

Note that exactly the same proof also works for $L^p$-functions for $1 \le p \le 2$.

\section{An application to orbital integrals}\label{sec:app}
Some of the motivation for this paper came from harmonic analysis on reductive groups, and we conclude it by an illustration: we apply the powerful techniques developed above to Fourier transforms of orbital integrals.

\subsection{Preliminaries}
Let $\bG$ be a connected reductive algebraic group
over a local field $F$, with Lie algebra $\fg$.
We need to represent $\bG(F)$ and $\fg(F)$ as members of a family of definable sets, and we do it as in
\cite[\S 2]{GordonHales}.  Let us briefly review this setting and the notation.

\subsubsection{Classification of reductive groups}\label{subsub:groups}
We start with the \emph{fixed choices} as in \cite[\S 2.1]{GordonHales}.
The fixed choices are constructed so that
for each Galois extension of local fields $L/F$ whose Galois group $\gal(L/F)=\Gamma$ matches the data from the fixed choice (see item (1) below) (and with residue characteristic of $F$ not $2$ or $3$),
a fixed choice $\Sigma$ determines (uniquely up to isomorphism) a split reductive group  $\bG^{\ast\ast}$ defined over $F$
and an action of $\gal(L/F)$ on $\bG^{\ast\ast}(L)$.

A fixed choice $\Sigma$ consists of:
\begin{enumerate}
\item An abstract finite group $\Gamma$ with an enumeration of its elements, a fixed subgroup $I$, and a distinguished element $\sigma_1$. 
Later when we associate a split reductive algebraic $F$-group $\bG^{\ast\ast}$ with our fixed choice, the group $\Gamma$ is interpreted as the Galois group of a Galois extension of $F$
with $I$ being the inertia subgroup, and  $\sigma_1$ is interpreted as a generator of the Galois group of the maximal unramified sub-extension;
\item a Dynkin diagram with an action of $\Gamma$ (this is interpreted as the Dynkin diagram of $\bG^{\ast\ast}$);
\item Two abstract lattices of the rank determined by the Dynkin diagram with action of $\Gamma$ (these are interpreted as the character and co-character lattices of a split maximal torus $\bT^{\ast\ast}$ in $\bG^{\ast\ast}$).
\end{enumerate}

Each fixed choice $\Sigma$ yields a finite set of isomorphism classes of connected reductive groups $\bG$ defined over $F$,
as described in \cite[\S 6]{GordonHales}. Those isomorphism classes corresponding to $\Sigma$ are distinguished by the elements of a definable parameter set $Z_\Sigma$ (which we will call the ``cocycle space''). An element of $(Z_\Sigma)_{\K}$ encodes (a) an extension $L/\K$ over which $\bG$ splits
(which has to have $\gal(L/F)\simeq \Gamma$ with the inertia subgroup mapping to $I$) and (b) the Galois cocycle that determines $\bG$ as a twist of $\bG^{\ast\ast}$.
We observe that if $\bG$ is unramified over $F$, then {$\bG$ itself} is also determined uniquely by the fixed choices.
The precise details of this construction will not be important here, and we refer the reader to \cite[\S 2]{GordonHales} for the details.

 Naturally, not all fixed choices of $\Gamma$, $I$, etc. give rise to a family of reductive groups (e.g. $\Gamma$ has to be solvable in order to be a Galois group of an extension of local fields, etc.). From here onwards, we are only interested in fixed choices that do satisfy the compatibility conditions that ensure they give rise to  a family of non-empty reductive groups.

Here it is important for us to note that even though the set $Z_\K$ is infinite for all $\K$, there are only finitely many isomorphism classes of connected reductive groups corresponding to the same fixed choices, as apparent by
inspection from \cite[\S 6]{GordonHales}, and these isomorphism classes are parameterized independently  of $\K$.
Indeed, in the case (1) of \emph{loc.cit.}, there are $d$ possible invariants in the Brauer group $\Q/\Z$ with denominator $d$; in the cases (2) and (3), excluding type $D_4$, there are two possible ramified quadratic extensions when $p>2$; if $\bG$ contains any factors of type $D_4$, there are still finitely many possibilities, corresponding to distinct  degree $3$ ramified extensions with Galois group $S_3$.

In summary, we have
\begin{prop}\label{prop:groups}(\cite[\S2.2.2]{GordonHales}, \cite[Theorem 4]{gordon-roe})
For every fixed choice $\Sigma$ there exists $M>0$, a definable set $Z_\Sigma$,
  and definable families
 $\dG \to Z_{\Sigma}$  and $\dg \to Z_\Sigma$ such that for every local field $F$ of residue characteristic greater than $M$
  the following holds:
 \begin{quote}
 for $z\in ({Z_\Sigma})_F$, the set $\dG_{F,z}$ is the set $\bG_z(F)$ of $F$-points of a connected reductive group
$\bG_z$ with absolute root datum determined by $\Sigma$
and the set $\dg_{F,z}$ is the set of $F$-points $\fg_z(F)$ of the Lie algebra of $\bG_z$.
\end{quote}
\end{prop}

In the following, we deal with the sets $\bG_z(F)$, $\fg_z(F)$, etc., while implicitly thinking of them as elements of the definable families
$\dG \to Z$ and $\dg \to Z$.
In particular,
by a ``definable subset $\omega \subset \fg_z(F)$'',
we really mean  a definable subset $\domega \subset \dg$, where $\omega := \domega_{F,z}$.

\subsubsection{Orbital integrals}\label{sub:orb.int}
Our goal is to provide uniform estimates for  Fourier transforms of the orbital integrals on the Lie algebra.

Let $X\in \fg(\K)$. An orbital integral at $X$ is the distribution
on $\tf(\fg(\K))$, defined by
$$\mu_X(f)=\int_{\bG(\K)/C_{\bG}(X)}f(\Ad(g^{-1}) X) d^\ast x,$$
where $C_{\bG}(X)=\{g\in \bG(\K)\mid \Ad(g)X=X\}$ is the centralizer of $X$, and
$d^\ast x$ is a $\bG(\K)$-invariant measure on $\bG(\K)/C_{\bG}(X)$.
Note that such a measure is unique up to a constant multiple, and is uniquely determined by the normalization of the Haar measures on $\bG(\K)$ and on $C_{\bG}(X)$.
The fact that orbital integrals converge and thus are well-defined distributions (which is a  non-trivial question when $X$ is not regular semi-simple) was proved by Deligne and Ranga Rao for the fields $\K$ of characteristic zero, and
by G. McNinch for $\K$ of sufficiently large positive characteristic.

In this paper, however, we focus on regular semisimple elements $X$. We will denote the set of regular semisimple elements in $\fg(\K)$ by $\fg(\K)^\rss$.
By definition, the centralizer $C_{\bG}(X)$ of a regular semisimple element is a maximal torus. There are finitely many conjugacy classes of maximal tori in $\bG(\K)$, with the upper bound on their number depending only on the fixed choices determining $\bG^{\ast\ast}$, so in order to pin down the normalization of the measures on all regular semi-simple orbits it suffices to pick a Haar measure on $\bG(\K)$ and on the (finitely many) representatives of conjugacy classes of the maximal tori.
The classical statements, quoted  below, on the boundedness of the Fourier transforms of orbital integral are independent of these normalizations of the measures.
However, the main point of this section is to discuss how these bounds depend on the field $\K$, and therefore we need a consistent normalization of measures.
Following e.g.\ Kottwitz' note on bounds for orbital integrals \cite[Appendix A]{ShinTemp}, we choose on $\bG(\K)$ and on all its maximal tori the so-called canonical measures, which assign volume $1$ to the canonical parahoric subgroup in the sense of Gross.
With this normalization of measures, for $X\in \fg(\K)^\rss$, let $\mu_X$ denote the orbital integral at $X$, considered as a distribution on $C_c^\infty(\fg(\K))$.

\subsubsection{Fourier transform}
Given an additive  character $\psi$ of $F$, and a non-degenerate $\bG(F)$-invariant bilinear form $B$ on $\fg(F)$, the Fourier transform for a
Schwartz-Bruhat function $f\in \tf(\fg(F))$ is defined as
$$\hat f(Y) =\int_{\fg(F)} f(X)\psi(B(X,Y)) \, dX.$$
Later we will allow the character $\psi$ to vary through the collection $\cD_F$ of level zero characters satisfying {Definition}  \ref{psiu}.
We will also need the bilinear form $B$ to be a definable function of $X$ and $Y$.
Such a form exists for all $F$ of characteristic zero or sufficiently large positive characteristic by
\cite{adler-roche}, see also \cite[\S 4.1]{CGH2}.

\subsubsection{The bounds on Fourier transforms of orbital integrals}
Let $D(X)$ be the discriminant  of $X$,
$$
D(X)=
\prod_{\alpha\in \Phi}\alpha(X),
$$
where $\Phi$ is the {root system} of $\bG$.

It is a well-known theorem of Harish-Chandra that for any $X\in \fg(\K)$, the orbital integral at $X$ (as a distribution on the space of locally constant compactly supported functions on $\fg(\K)$) is represented by a function $\widehat\mu_X$, defined and locally constant on the set of regular semisimple elements $\fg(\K)^\rss$; this function, extended by zero to all of $\fg(\K)$, is locally integrable on
$\fg(\K)$. Moreover, the function
$|D(Y)|^{1/2}\widehat\mu_X(Y)$ is locally bounded on $\fg(\K)$.
Following \cite{kottwitz:clay}, we will call such a function a ``nice'' function on
$\fg(\K)$.

R. Herb generalized the boundedness result, also
making it uniform in $X$ in some sense \cite{herb}.
To state Herb's theorem precisely, we
first need a  definition.
Let $\fh$ be a Cartan subalgebra of $\fg(\K)$, let $T$ be the torus in $\bG(\K)$ such that
$\fh= \Lie(T)$, and let $A$ be the split component of $T$.
Fix an arbitrary open compact subgroup $K$ of $\bG(\K)$, and
define a function on  $\fh(\K)\times \fg(\K)$ by
\begin{equation}\label{eq:Phi:0}
\Phi(\fg, X,Y):= |D(X)|^{1/2}|D(Y)|^{1/2}\int_{\bG(\K)/A}\int_K \psi(B(kY, xX))\,dk dx^\ast,
\end{equation}
where $dx^\ast$ is the quotient measure obtained from the Haar measures on $\bG(\K)$ and on $A$, and $dk$ is the normalized Haar measure on $K$.
This definition is due to Harish-Chandra, and in his definition the normalization of the measure on $A$ does not matter.  We note that $T/A$ is compact; we have already chosen a normalization of the measure on $T$. Let us normalize the measure on $A$ so that the volume of $T/A$ is $1$.
This gives a collection of measures $dx^\ast$ on $\bG(\K)/A_i$, where $A_i$ are the
split components of the representatives of the conjugacy classes of maximal tori in
$\bG(\K)$.

For any fixed collection of normalizations of measures on $\bG(\K)/A_i$, consistent with the normalization of measures on the maximal tori,
Harish-Chandra proved the following
\begin{lem}(\cite[Lemma 7.9]{hc:queens}) Let $\fh$ and $\fh_1$ be two Cartan subalgebras in $\fg(\K)$. For $X\in \fh^\rss$, $Y\in \fh_1^\rss$,
$$\Phi(\fg, X,Y)= |D(X)|^{1/2}|D(Y)|^{1/2}\widehat\mu_X(Y).$$
\end{lem}
R. Herb proved that the function  $\Phi(\fg, X,Y)$ is uniformly bounded on compact sets:
\begin{thm}(\cite[Theorem 1.3]{herb})\label{thm:herb}
 Let $\fh$ be a Cartan subalgebra of $\fg$, and let $\omega$ be a compact subset of the set of regular elements in $\fh$.
Then
$$
\sup_{X\in \omega, Y\in \fg^\rss}|\Phi(\fg, X, Y)|<\infty.
$$
\end{thm}

Combining this theorem with Harish-Chandra's lemma, one obtains that Fourier transforms of regular semisimple  orbital integrals are uniformly bounded on compact sets, in a given Lie algebra over a given field.

\subsection{Uniformity: the questions}
Our goal is to answer a series of related natural  questions:

\begin{enumerate}
\item How does the bound in Theorem \ref{thm:herb} vary for a family of compact sets $\omega_\lambda$?
\item Given a definable set $\omega$, how does the bound in Theorem \ref{thm:herb}
depend on the field $\K$?
More specifically, do the same bounds apply for the fields $\K$ of positive characteristic and of characteristic zero with isomorphic residue fields? And,
\item Given a definable set $\omega$, how does the bound depend on the cardinality of the residue field as $\K$ varies through the family of all completions of a given global field?
\end{enumerate}

\subsection{Uniform bounds for Fourier transforms of regular semisimple orbital integrals}

Here we prove two results, analogous to Theorems 1 and 2 of
\cite[Appendix B]{ShinTemp}, for the Fourier transforms of regular semisimple orbital integrals, which answer the above questions.

\begin{thm}\label{thm:orb.int.bound}
Let $\Sigma$ be a fixed choice as in \S \ref{subsub:groups}. Let $Z_\Sigma$ be the corresponding cocycle space, and let
$\dG\to Z_{\Sigma}$
be the definable family of reductive groups as in Proposition \ref{prop:groups}, with $\dg\to Z_\Sigma$ the corresponding family of Lie algebras.
Let $\{\omega_\lambda\}_{\lambda\in \Z^n}$ be a definable family of subsets
of $\fg$, understood as in the remark after Proposition \ref{prop:groups}.
Then
\begin{enumerate}
\item There exist constants $M>0$,  $a$ and $b$ that depend only on $\Sigma$ and on the {formula} defining
$\{\omega_\lambda\}_{\lambda\in \ZZ^n}$,  such that for each non-Archimedean local field $\K$
with residue characteristic at least $M$, the following holds. Let $\bG$ be any connected reductive algebraic group over
$\K$ corresponding to the fixed choice $\Sigma$.
Then
for all $X\in \omega_\lambda$ with $\|\lambda\|\le \kappa$, for all $Y\in \fg(\K)^\reg$,
\begin{equation}\label{eq:bound}
|D(X)|^{1/2}|D(Y)|^{1/2}|\widehat\mu_X(Y)|\le
{q_\K}^{a+b\kappa}.
\end{equation}
\item Moreover, if some constants $a$ and $b$  provide a bound for all fields of sufficiently large positive characteristic, then there exist constants $M$ and $N$ (as above, depending only on the data defining the group and the family of subsets $\{\omega_\lambda\}$) such that  for all local fields of characteristic zero and residue characteristic larger than $M$,
the bound
\begin{equation}
|D(X)|^{1/2}|D(Y)|^{1/2}|\widehat\mu_X(Y)|\le
N{q_\K}^{a+b\kappa}.
\end{equation}
holds for all $X\in \omega_\lambda$ with $\|\lambda\|\le \kappa$, for all $Y\in \fg(\K)^\reg$, and vice versa, i.e., with positive characteristic and characteristic zero swapped.
\end{enumerate}
\end{thm}

In order to prove this theorem, we need  a  lemma establishing that orbital integrals \emph{with respect to the canonical measure} are motivic.
The fact that orbital integrals are motivic functions was first proved for \emph{semisimple orbital integrals} in \cite{CHL}, but only for unramified reductive groups. This result was extended to all orbital integrals and all connected
reductive groups in \cite{CGH2}, but the measure on the orbits used there was not necessarily the canonical measure.
Further, this lemma was needed for all semisimple elements, not necessarily regular, in \cite[Appendix B]{ShinTemp}, but at that time we could only prove it up to a uniformly bounded factor (denoted by $i_M$ in \emph{loc. cit.}).
Now we include the proof  in full generality for completeness, even though in this paper we need the statement only for regular semisimple elements.

First, we need one more notation. We observe that the connected component of a centralizer of a semisimple element in $\fg(F)$ is the set of $F$-points of a connected reductive algebraic group, and there is a finite list of possible root data associated with such groups, cf. \cite[\S A.2]{ShinTemp}. In particular, there are finitely many fixed choices $\Sigma_i$ giving rise to split forms
$\bM^{\ast\ast}$ of such centralizers, and for each
$\Sigma_i$, a corresponding cocycle space $Z_i'$ parameterizing all their forms.
We let $Z'$ be the disjoint union of all these cocycle spaces (considered as a definable set); this yields a definable family $\dM \to Z'$ which has the property that all centralizers of elements of $\fg(F)$ arise as $\bM_{z'}(F) := \dM_{F,z'}$ for some $z' \in Z'_F$.
We also note that the condition on the pair $(X, z')\in \fg(F)\times Z'_F$ stating that the centralizer of $X$ is isomorphic to
$\bM_z'$ as reductive groups over $F$ is  a definable condition.
\begin{lem}\label{lem:can.orb}
Let $\dG\to Z_\Sigma$ be as in Theorem \ref{thm:orb.int.bound} above, and let $\dM\to Z'$ be a family of subgroups of $\bG_z$ as above.
Then there exists an integer $M>0$ and a function
$c^{\bG}$ in $\cC(Z)$, such that for any definable family $\{f_a\}_{a\in S}$ of test functions on $\fg_z$
of $\cCexp$-class
there exists a function $H(X, a, z,z')$ in
$\cCexp(\dg\times S\times Z\times Z')$ such that for
all {local fields of residue characteristic greater than $M$}, the following holds.

For every $z\in (Z_\Sigma)_F$, let $\bG_z$ be the corresponding connected reductive group with Lie algebra $\fg_z$.
Let $X \in \fg_z(F)$ be a semisimple element
with centralizer isomorphic to $\bM_{z'}$ for some $z'\in Z'_F$.
Then
$$\mu_X(f_a)=\frac{1}{{c^{\bG}_F}(z)}H_F(X, a,z,z').$$
We recall from \S \ref{sub:orb.int} that $\mu_X$ denotes the orbital integral at $X$ with respect to the canonical measure.
\end{lem}

\begin{rem} We note that everywhere in this lemma, we could replace $\cCexp$-class with $\cC$-class, see Remark \ref{remove-exp}.
\end{rem}

\begin{proof}
Recall that we deal with the sets $\bG_z(F)$, $\fg_z(F)$, etc., while thinking of them as elements of a definable family:
$\bG_z(F) = \dG_{F,z}$, for any $z\in (Z_\Sigma)_F$.
Since the centralizer of $X$ is a definable set (with its defining formulas using $X$ as a parameter), we see that
the condition that the centralizer of $X$ is isomorphic to $\bM_{z'}(F)$ for some $z'\in Z'_F$
gives a definable subset of $\dg\times Z'$. In particular, there exists a definable family  of
subsets
$\fg_{z, z'}$ of $\dg$,
such that
$$({\fg_{z,  z'}})_F =\{X\in \fg_z(F)\mid C_{\bG_z}(X)\cong \bM_{z'} \text{ as reductive groups over $F$}\},$$
where $C_{\bG_z}(X)$ denotes the centralizer of $X$ under the adjoint action of $\bG_z(F)$.

By \cite[Theorem 6]{gordon-roe}, there exists a function $c^\bG\in \cC(Z)$, and a family of motivic measures
$\lambda_z$ on the groups $\bG_z$ such that the canonical measure $dg$ on $\bG_z(F)$
satisfies
$$c_G(z) dg = d\lambda_z(g).$$
Now, if $\frac{dg}{dt}$ is the quotient measure on the orbit of $X$, where $dt$ is the canonical measure on
$C_{{\bG}_z(F)}(X)$, then we can represent it as
$$\frac{dg}{dt} = \frac{c^{\bM}(z'){d\lambda_z}}{c^\bG(z)d\lambda_{z'}}=
\frac1{c^\bG(z)}\frac{c^{\bM}(z'){d\lambda_z}}{d\lambda_{z'}}.$$
Consider the quotient measure $\frac{c^{\bM}(z'){d\lambda_z}}{d\lambda_{z'}}$.
It is a quotient of two motivic measures, and thus it is motivic by exactly the same argument as the one of \cite[Lemmma 14.15]{ShinTemp}.
The statement of the lemma follows.
\end{proof}

\subsubsection{Proof of Theorem \ref{thm:orb.int.bound}} Let $X\in \fh(K)$ be an arbitrary regular element, and let $T=G_X$ be its centralizer (which is a torus independent of $X$, with $\Lie(T)=\fh$).
First, note that the statement of Theorem \ref{thm:herb} is independent of the normalization of the measure on $T$.
More precisely, it is independent of the normalization of the measure on $A$, the maximal split subtorus of $T$. We can choose such normalization that the volume of the compact torus $T/A$ is equal to $1$.
Then we  have
\begin{equation}\label{eq:Phi}
\Phi(\fg, X, Y)=|D(X)|^{1/2}|D(Y)|^{1/2}\int_{\bG(\K)/T}\int_K \psi(B(kY, xX))\,dk dx^\ast,
\end{equation}
where $dx^\ast$ is the resulting quotient measure.
In particular, for the choice of measure $dx^\ast$ that corresponds to the canonical measures (as above) on
$\bG(\K)$ and on $T$, we obtain, by Theorem \ref{thm:herb} that $\Phi(\fg, X, Y)$ is bounded on $\omega\times\fg^\rss$.

By Lemma \ref{lem:can.orb}, and since $X \mapsto |D(X)|^{1/2}$ is a
function of $\cC$-class (generalized to allow for $\sqrt{q}$ as in  \cite[B.3.1]{CGH2}),
the right-hand side of (\ref{eq:Phi}) is ``of $\cCexp$-class up to $\cC$-constant'',
i.e., we can write (\ref{eq:Phi}) in the form
$$\Phi(\fg_{z}, \fh_{z'},  X, Y)= \frac{1}{{c^{\bG}_F}(z)}H_F(X, Y, z, z'),$$
where $H$ is a $\cCexp$-class function (also generalized as in \cite[B.3.1]{CGH2}),
$z\in Z$, and
the parameter $z'$
determines the isomorphism class of the centralizer of $X$. Since here $X$ is assumed to be regular, its centralizer is a torus, and $z'$
corresponds to the choice of the Cartan subalgebra $\fh$ that contains $X$.

Let us consider the function
$$\tilde H(X, Y, z, z', \lambda):= {\bf 1}_{\omega_\lambda}H(X, Y, z, z').$$

By Theorem \ref{thm:herb}, we know that for every $\lambda\in \ZZ^n$ and every pair $(z, z')\in Z_F\times Z'_F$, the function
$\Phi(\fg_z, \fh_{z'}, X, Y)$ is bounded (as a function of $X$ and $Y$) on $\omega_\lambda\times \fg$.

Next we use the observation from \S \ref{subsub:groups} that for every fixed choice  there are, in fact, finitely many
possible values of $z, z'$ that give rise to distinct Lie algebras. Hence, for every $\lambda$, the function
$\tilde H(z, z', X, Y, \lambda)$ is bounded.
Then by Theorem \ref{thm:presburger-fam}, we have
\begin{equation}
|\tilde H_F(X, Y, z, z',\lambda)|< q^{a +b\|\lambda\|},
\end{equation}
where the integers $a$ and $b$ depend only on the fixed choices, which completes the proof of Part (1).
Indeed, since we are working in sufficiently large residue field characteristic, the remark just below Theorem \ref{thm:presburger-fam}
allows us to use an integer $a$ instead of a definable $d\in \VG$ in the exponent.

 {\bf Part (2).} Follows immediately from \cite[Theorem 3.2]{CGH4}, with a single function $H$ from Part (1) on the left, and
$G(X, Y, z, z',\lambda)= q_F^{a +b\|\lambda\|}$.

\medskip

Finally, in the spirit of \cite[Appendix B]{ShinTemp}, we state an easy corollary that might be useful.
\begin{thm}
Let $G$ be a connected reductive algebraic group over  a number field ${\mathbf F}$.
Let $\{\omega_\lambda\}_{\lambda\in \ZZ^n}$, be a definable family of definable subsets
of $\fg({\mathbf F}_v)$.
Then there exist constants $a_G$ and $b_G$ that depend only on the
global model of $G$ and the formulas defining the subsets $\omega_\lambda$
such that for all {$\lambda \in \ZZ^n$} with $\|\lambda\|\le \kappa$, for
all but finitely many places $v$,
for all $X\in \omega_\lambda$ with $\|\lambda\|\le \kappa$, for all $Y\in \fg({\mathbf F}_v)^\rss$,
$$
|D(X)|^{1/2}|D(Y)|^{1/2}|\widehat\mu_X(Y)|
  \le q_v^{a_G+b_G\kappa} $$
where $q_v$ is the cardinality of the residue field of ${\mathbf F}_v$.
\end{thm}
\begin{proof} Since there are finitely many possibilities for the root data of the groups $G_v$ as $v$ varies over the finite
places of $\mathbf F$, this theorem immediately follows from Theorem \ref{thm:orb.int.bound}, exactly in the same way as Theorem
14.1 of \cite{ShinTemp} follows from Theorem 14.2.
\end{proof}

\appendix

\renewcommand{\thethm}{\thesection.\arabic{thm}}

\section{Adding constants to the language}
\label{sec:const}

For simplicity, all results in this paper are stated using a language containing only basic constant symbols.
However, in all our results, more constants can be added to the language for free, by usual model theoretic arguments and using
that all our results are stated ``in families''; the idea is that any constant can be replaced by a
family parameter. (Note however that the transfer principles need some extra care; see below.)
For example, one can fix a ring $A$, assume that each local field $F$ is equipped with an injection
$\iota_F\colon A \to F$ and that the language contains constant symbols for the elements of $\iota_F(A)$,
or we can assume that each $F$ comes with a uniformizer for which the language contains a constant symbol.
Since not all readers might be familiar with this, we recall how this works.

Fix a set of (new) constant symbols $\Omega$ (in any of the sorts) and set $\ldpo{\Omega} := \ldp \cup \Omega$ (where $\ldp$ is the
language introduced in Definition~\ref{defn.lang}).
We write $\Loc^{\rm all}_{\Omega}$ for the set of fields $F \in \Loc$ which come equipped with (arbitrary) interpretations of
the constant symbols in $\Omega$. Fix an arbitrary subset $\Loc_{\Omega} \subset \Loc^{\rm all}_{\Omega}$, and
define $\Loc_{\Omega,M}$ and $\Loc_{\Omega,\gg 1}$ analogously to Definition~\ref{AO}.
(Usually, in model theory, one would require the subset $\Loc_{\Omega}\subset \Loc^{\rm all}_{\Omega}$ to be defined
by a set of first order sentences, but here, we do not even need that.)


\begin{prop}\label{prop.const}
Each of the results listed in Section~\ref{sec:summary}, with exception of the transfer results, also hold if one replaces $\ldp$ by $\ldpo{\Omega}$ and $\Loc$ by $\Loc_{\Omega}$ everywhere.
\end{prop}

\begin{proof}
The results from Section~\ref{sec:app} build on the ones from the previous sections, so it suffices to verify that
the proposition applies to the results from those previous sections.
First note that in each of those, $\Loc_{\Omega}$ only appears in the form ``for each $F \in \Loc_{\Omega,\gg 1}$''
(or variants of that), so in particular, we may as well assume $\Loc_{\Omega} = \Loc^{\rm all}_{\Omega}$.

Next, since each theorem involves only a finite number of ``input objects'' (definable sets, definable functions,
$\cCexp$-function, etc.) and each input object involves only finitely many formulas,
only finitely many of the constants from $\Omega$ are used, so we may as well assume that we only added
a finite tuple $\omega$ of constants to the language.

Finally, note that each result involves a parameter set (always denoted by $X$).
We introduce a new tuple $z$ of variables of the same sorts as $\omega$, we let
$Z$ be the corresponding cartesian product of sorts, we replace the parameter set $X$
by $Z \times X$, and in all input objects, we replace all occurrences of $\omega$ by $z$.
Then we apply the $\ldp$-version of the theorem.
If the theorem has some ``output objects'' (this is the case for all results but Lemma~\ref{Lp-vs-pointwise}),
then in that output object, we replace the $z$ by $\omega$ again. The $\ldpo{\Omega}$-objects obtained
in this way have the desired properties.
\end{proof}

For the transfer principles to hold with additional constants in the language, one needs extra control on these constants via suitable axioms specifying the subset $\Loc_{\Omega} \subset \Loc^{\rm all}_{\Omega}$.
Instead of trying to formulate this in the most general possible way, we just formulate it for
collections of constants that are commonly used (e.g.\ in \cite{CGH}): Let $\cO$ be the ring of integers of some fixed number field, set
$\Omega := \cO[[t]]$, and let $\Loc_{\Omega}$ be the set of fields $F \in \Loc$ together with a ring homomorphism
$\iota_F\colon \cO \to F$. Recall that an $F \in \Loc$ comes with a uniformizer $\varpi_F$.
We consider an $F \in \Loc_{\Omega}$ as an $\ldpo{\Omega}$-structure by interpreting elements of $\Omega$ via
the continuous ring homomorphism which is equal to $\iota_F$ on $\cO$ and which sends $t$ to $\varpi_F$.
Given an element $\omega \in \cO[[t]]$, we write $\omega_F$ for its interpretation in $F$.

\begin{prop}\label{prop.const.trans}
Let $\Omega$ and $\Loc_{\Omega}$ be as described right above this proposition. Then
each transfer result listed in Section~\ref{sec:summary} still holds if one replaces $\ldp$ by $\ldpo{\Omega}$ and $\Loc$ by $\Loc_{\Omega}$ everywhere.
\end{prop}
\begin{proof}
Theorem~\ref{thm.trans}, which we cited from \cite{CLexp}, is actually stated in the language $\ldpo{\Omega}$ and for fields in $\Loc_\Omega$ in \cite{CLexp},
and follows from Denef--Pas quantifier elimination.
Since all other transfer results are deduced from Theorem~\ref{thm.trans} (and from results to which Proposition~\ref{prop.const} applies), they hold accordingly.
\end{proof}

Finally, let us make precise the comments on $d$ below Theorems \ref{thm:presburger-fam} and \ref{thm:fam:gen}. If $\Omega$ and $\Loc_{\Omega}$ are as described right above Proposition \ref{prop.const.trans}, then any $\ldpo{\Omega}$-definable $d\in\VG$ can still be bounded by $a+\ord (c)$ for some integers $a$ and $c$ depending on $d$. However, for more general $\Omega$ and $\Loc_{\Omega}$ such integers $a$ and $c$ may no longer yield bounds for $d_F$ uniformly in $F$.

\begin{remark}
All results in this paper also stay true if one consequently expands the language by some analytic structure as in \cite{CLip,CLips}, or by arbitrary additional structure on the residue ring sorts $\RF_n$. Indeed, by \cite[Section 4.7]{CHallp}, this is true for the results from \cite{CHallp} the present paper is based on.
\end{remark}

\bibliographystyle{amsplain}
\bibliography{anbib}
\end{document}